\documentclass{article}
\usepackage{graphicx} % Required for inserting images
\usepackage{amsmath,amssymb,mathrsfs}
\usepackage{xcolor}
\usepackage{amsthm}
\usepackage[colorlinks=true,linkcolor=blue,citecolor=red,urlcolor=black]{hyperref}

\def\bel{\begin{equation*} \begin{aligned}}
\def\eel{\end{aligned} \end{equation*}}
\def\beln{\begin{equation} \begin{aligned}}
\def\eeln{\end{aligned} \end{equation}}
\def\beq{\begin{equation}}
\def\eeq{\end{equation}}
\newcommand{\bea}{\begin{eqnarray}}
\newcommand{\eea}{\end{eqnarray}}

\newtheorem{theorem}{Theorem}[section]

\newtheorem{lemma}[theorem]{Lemma}

\newtheorem{remark}[theorem]{Remark}
\newtheorem{definition}[theorem]{Definition}

\newcommand\fB{\mathfrak{B}}
\newcommand\cO{\mathcal{O}}

\begin{document}

\centerline{\bf \Large Couette Taylor instabilities in the small-gap regime}

\bigskip

\centerline{D. Bian\footnote{School of Mathematics and Statistics, Beijing Institute of Technology, Beijing 100081, China. Email: biandongfen@bit.edu.cn and emmanuelgrenier@bit.edu.cn}, 
E. Grenier$^{1}$,
G. Iooss\footnote{Laboratoire J.A.Dieudonné, I.U.F., Université Côte d’Azur, Parc Valrose, 06108 Nice Cedex 02, France},
Z. Yang\footnote{Department of Mathematics, The Ohio State University, Columbus, OH 43210, USA. Email: yang.8242@osu.edu}
}

%%%%%%%%%%%%%%%%%%%%

\subsubsection*{Abstract}

The Couette-Taylor instability occurs in  a viscous fluid confined between two coaxial rotating cylinders. When the Taylor number surpasses a critical value, the stable Couette flow destabilizes, giving way to steady Taylor vortices. As the Taylor number increases further, these vortices themselves become unstable, transitioning into wavy Taylor vortices.

In this article, we focus on the small-gap limit, where the ratio of the cylinder radii approaches unity and the rotation rates of the cylinders are nearly identical. We provide a rigorous proof of the existence of a critical Taylor number $T_c$, at which the Couette flow loses stability.

For Taylor numbers just above $T_c$ —under fixed axial periodicity—the solutions to the limiting Navier-Stokes system are governed by a Ginzburg-Landau-type partial differential equation. Beyond the classical Taylor vortex flow, we demonstrate that a two-parameter family of solutions emerges at criticality for $T>T_c$. This family includes not only wavy vortices but also a variety of other exotic flow patterns, all of which remain steady in the frame rotating at the average angular velocity of the cylinders.

%%%%%%%%%%%%%%%%%%%%

\tableofcontents

%%%%%%%%%%%%%%%%%%%%%%%%%%%%%%%%%%%%%%%%%%%%%%%%%

\section{Introduction}

%%%%%%%%%%%%%%%%%%%%%%%%%%%%%%%%%%%%%%%%%%%%%%%%%

In this article, we study the classical problem of the flow of a viscous fluid between two rotating cylinders,
motivated in particular by the works of M. Nagata \cite{Nagata86}, \cite{Nagata23}, \cite{Nagata24}.
More precisely, we consider the Navier-Stokes equations
\begin{equation} \label{NS1}
\partial_t u + (u \cdot \nabla) u - \nu \Delta u + \nabla p = 0,
\end{equation}
\begin{equation} \label{NS2}
\nabla \cdot u = 0
\end{equation}
in the domain 
$$
\Omega = \Bigl\{ (r,\varphi,z) \quad  | \quad 
r_i < r < r_o \Bigr\}.
$$
In these equations, $u(t,x,y,z)$ denotes the velocity of the fluid, $p(t,x,y,z)$ the pressure,
and $\nu > 0$ the viscosity. Moreover, $(r,\varphi,z)$ are the cylindrical coordinates.
The domain $\Omega$ is the interval between two cylinders with the same axis at $r = 0$,
the inner one being of radius $r_i$ and the outer one being of radius $r_o$.

The inner cylinder (respectively the outer) rotates with an angular velocity
$\omega_i$ (respectively, $\omega_o$). We assume that the fluid ``sticks" to the boundary, namely that the velocity of the fluid equals the velocity
of the cylinders at $r = r_i$ and $r = r_o$.

Let us define the gap $d$, the radii ratio $\eta$ and the ratio of rotation rates $\mu$ by
$$
d = r_o - r_i, \quad \eta = \frac{r_i}{r_o}, \quad \mu = \frac{\omega_o}{\omega_i}.
$$
We work in a rotating frame, with a rotation rate $\Omega_{rf}$ defined by
$$
\Omega_{rf} = \frac{\omega_i + \omega_o}{2} =  \omega_i \frac{1+\mu}{2}.
$$
We consider the stability and bifurcations of the Couette flow, defined by
\begin{equation} \label{Couette}
U(r) = r \, \Omega_{rot} (r) = (A - \Omega_{rf})r + \frac{B}{r},
\end{equation}
in this rotating frame, where $A$ and $B$ are constants such that the velocity of the fluid coincides with 
the velocity of the cylinders at the inner and outer cylinders
$$
A = \omega_i \frac{\mu - \eta^2}{1 - \eta^2},
\quad 
B = \omega_i r_i^2 \frac{1 - \mu}{1 - \eta^2}.
$$
We focus on the small gap / slow rotation / high Reynolds regime where
\begin{equation}\label{limit}
    \eta \to 1, \qquad \mu \to 1,
\qquad
\hat{\omega}=\omega_i \frac{d^2}{\nu}\to 0, \qquad
{\mathfrak R} = \frac{ \hat \omega (1 - \mu)}{1 - \eta}
\to \infty ,
\end{equation}
and we define the Taylor number to be
$$
T = 2 \hat{\omega} \mathfrak{R}.
$$
We refer to section \ref{sec2} for the adimensionalization of the Navier-Stokes equations.

We first consider axi-symmetric perturbations and prove the existence of a critical Taylor number $T_c$.

\begin{theorem}
For $\eta$ and $\mu$ close enough to $1$, there exists $T_c > 0$ such that the Couette flow (\ref{Couette}) is linearly stable
with respect to axisymmetric perturbations if $T < T_c$ and linearly unstable if $T > T_c$.
\end{theorem}

Let $\alpha$ be the Fourier wave number associated with the axial coordinate $z$. Using the method of oscillating kernels, 
without taking our asymptotics,
V. Yudovich \cite{Yudo1}  proved that for every $\alpha \ne $0, there exists a critical Taylor number
$T_c(\alpha)>0$. 
 This number defines the stability threshold for the Couette flow with respect to axisymmetric perturbations:
the flow is stable for $T < T_c(\alpha)$ and unstable for $T > T_c(\alpha)$.
Furthermore, V. Yudovich proved that  the function $T_c(\alpha)$ is even, analytic, and goes to infinity as
$\alpha$ approaches to $0$ or to $\pm \infty$. Consequently, the critical Taylor number
$$
T_c = \min_{\alpha \in \mathbb{R}} T_c(\alpha)>0
$$
is well defined. However, it remains an open question whether this minimum is achieved at only two critical wavenumbers, $\pm \alpha_c$, or if additional values of $\alpha$ may also yield the minimum.

Thanks to our triple limit $\eta \to 1$, $\mu \to 1$, $\hat{\omega} \to 0$,
the analysis of linear stability is significantly simplified. In this regime, the theory of oscillating kernels is no longer required, and the problem reduces to determining whether a certain self-adjoint operator admits a positive eigenvalue. Specifically, it suffices to investigate whether $0$ is an eigenvalue of a $6 \times 6$ system of ordinary differential equations, whose solutions can be computed explicitly. For a fixed $\alpha$, $T_c(\alpha)$ is characterized as the smallest positive root of an explicit hyperbolic-trigonometric function, ensuring its existence (see Lemma \ref{lem3.2}). Numerically, it is straightforward to tabulate $T_c(\alpha)$ as the zeros of this explicit function. Our calculations reveal that the global minimum $T_c \approx 1708$
is achieved at only two wavenumbers, $\alpha = \pm \alpha_c \approx \pm 3.117$, in agreement with classical results  \cite{Drazin}.

\medskip

We then turn to non axi-symmetric perturbations. 
Let $\beta$ denote the azimuthal wave number.

\begin{theorem}
For fixed axial wavelength, if $|\beta \mathfrak{R}|$ small enough, then the set of eigenvalues of the linearized operator near the Couette flow is composed of a denumerable set of \emph{real intervals}. The largest eigenvalue $\lambda$ is double. Moreover, for $T>T_c$, $\lambda$ crosses $0$ together with an interval of real eigenvalues of length $T-T_c-b(\beta \mathfrak{R})^2$ ($b>0$).    
\end{theorem}
The theorem follows from Lemmas \ref{lem4.1} and \ref{lem4.2} and is based on the detailed analysis
of the eigenvalue $\lambda_0$ with the largest real part.
We prove that $\lambda_0$ is real, and that 
for $T$ close to $T_c$, $\alpha$ close to $\alpha_c$, and $\beta \mathfrak{R}$ close to $0$,
$$
\lambda_0 =  a_3 (T - T_c) + a_4 (\beta \mathfrak{R})^2 + a_6 (\alpha^2 - \alpha_c^2)^2 +  \cdots .
$$
where $a_3 > 0$, $a_4 <0 $ (proved numerically) and $a_6 <0$.

\begin{remark}
    It is a remarkable fact that all eigenvalues are real, despite the non-self-adjointness of the linearized operator. This is due to an additional symmetry with respect to the $z$ axis that commutes with the system (see Lemma \ref{lem4.1}). This is specific to our limit case. Moreover, at criticality the eigenvalue $0$ is not isolated, which prevents us from using classical reduction methods for bifurcation problems.
\end{remark}

As $\eta \to 1$, the azimuthal period tends to infinity, so the azimuthal angle $\varphi$ 
may be treated as a real variable $y \in \mathbb{R}$ (see section $2$ for the detailed computations).
In this limit, the initial Navier-Stokes system turns into a new system (\ref{newnonlinear}) that we
call the limiting Navier-Stokes system.

For $T> T_c$ with $T$ close to $T_c$, the Couette flow is unstable with respect to perturbations 
with small azimuthal wave number $\beta$. 
For the limiting Navier–Stokes system (\ref{newnonlinear}), 
the dynamics of this nonlinear instability can be formally described by a
Ginzburg-Landau equation, written on an amplitude $A(t,y)$, which satisfies
\begin{equation} \label{GLequ1}
\frac{\partial  A}{\partial t}=a_{3}(T - T_c) A- a_{4}\mathfrak{R}^2 \frac{\partial^2 A}{\partial y^2} -cA|A|^2,
\end{equation}
where $c$ is a positive real number.
We do not attempt to justify (\ref{GLequ1}) here. Instead, we focus on the time independent Ginzburg-Landau equation
\begin{equation} \label{GLequind}
a_{3}(T-T_c) A- a_{4}\mathfrak{R}^2 \frac{\partial^2 A}{\partial y^2} -cA|A|^2 = 0.
\end{equation}
Using the methods of spatial dynamics, in particular the approach of \cite{Iooss2}, it can be proved
that small-amplitude stationary solutions of the limiting Navier-Stokes system (\ref{newnonlinear}) are,
at leading order, described by small solutions of (\ref{GLequind}). The argument in \cite{Iooss2}, developed for
a nearby problem, carries over to our current problem with only minor and simple modifications. 
To avoid reproducing a lengthy classical proof,
we omit this justification and focus instead on the analysis of (\ref{GLequind}).
The next Theorem describes the small-amplitude solutions of (\ref{GLequind}).

\begin{theorem}\label{maintheo}
For $T>T_c$, in addition to the constant solutions, 
a one parameter family of solutions of (\ref{GLequind}) bifurcates supercritically. For these solutions, the amplitude $A$ has a constant modulus $|A|$ (independent of $y$) and a phase that is linear in $y$. 

Moreover, a two-parameters family of exotic solutions  of (\ref{GLequind}) bifurcates, with a modulus 
$|A|$ which oscillates periodically in $y$, and a phase that is the sum of an oscillating part 
and a linear part (see Figures \ref{diagram} and \ref{rhotheta}).
\end{theorem}

In the theorem above, the constant solution corresponds to the steady axi-symmetric Taylor vortex flow. The one-parameter family
of solutions corresponds to three dimensional wavy vortices, which are steady in the rotating frame.
The two-parameters family corresponds to exotic solutions of the limiting Navier-Stokes equations (\ref{newnonlinear}),
whose amplitude oscillates periodically in $y$.

%%%%%%%%%%%%%%%%%%%%%%%%%%%%%%%%%%%%%%%%%%%%%%%%%%%%%%%%%%%%

\section{The Navier-Stokes equations in cylindrical coordinates}\label{sec2}

%%%%%%%%%%%%%%%%%%%%%%%%%%%%%%%%%%%%%%%%%%%%%%%%%%%%%%%%%%%%

In this section, we adimensionalize the Navier-Stokes equations in cylindrical coordinates.
In cylindrical coordinates, a (not necessarily axi-symmetric) perturbation $(u_r, u_\varphi, u_z, p)$ of the Couette flow \eqref{Couette} satisfies the system (see \cite{Nagata23})
\begin{align*}
   \Big( \frac{\partial}{\partial t} - \nu (\Delta_{cyl} - &\frac{1}{r^2}) \Big) u_r 
   + \nu \frac{2}{r^2} \frac{\partial u_\varphi}{\partial \varphi} 
   + \frac{\partial p}{\partial r} 
   \\
   &= - \Omega_{rot} \frac{\partial u_r}{\partial \varphi} + 2\Omega_{rot} u_\varphi  
   + \frac{u_\varphi^2}{r} + 2\Omega_{rf} u_\varphi - (u \cdot \nabla) u_r,
   \\
   \Big( \frac{\partial}{\partial t} - \nu (\Delta_{cyl} &- \frac{1}{r^2}) \Big) u_\varphi 
   - \nu \frac{2}{r^2} \frac{\partial u_r}{\partial \varphi} + \frac{1}{r} \frac{\partial p}{\partial \varphi} 
   \\
   &= - \Omega_{rot} \frac{\partial u_\varphi}{\partial \varphi} { - U' u_r - \Omega_{rot} u_r} 
   - \frac{u_r u_\varphi}{r} - 2\Omega_{rf} u_r  - (u \cdot \nabla) u_\varphi,
   \\
\Big( \frac{\partial}{\partial t} - &\nu \Delta_{cyl}  \Big) u_z 
+ \frac{\partial p}{\partial z} 
= - \Omega_{rot} \frac{\partial u_z}{\partial \varphi}  - (u \cdot \nabla) u_z,
\end{align*}
together with the incompressibility condition, which in
cylindrical coordinates reads
$$
\frac{\partial u_r}{\partial r} + \frac{u_r}{r} + \frac{1}{r} \frac{\partial u_\varphi}{\partial \varphi} 
+ \frac{\partial u_z}{\partial z} = 0.
$$
Here the Coriolis term is added on the right-hand side,
and $\Delta_{cyl}$ is the Laplace operator in
cylindrical coordinates, namely
$$
\Delta_{cyl} = \frac{\partial^2}{\partial r^2} + \frac{1}{r} \frac{\partial}{\partial r} 
+ \frac{1}{r^2} \frac{\partial^2}{\partial \varphi^2} + \frac{\partial^2}{\partial z^2}.
$$
We perform the change of variables
$$
\hat{x} = \frac{r - \bar{r}}{d}, \quad \hat{y} = \frac{\bar{r}}{d} \varphi, 
\quad \hat{z} = \frac{z}{d}, \quad \hat{t} =  \frac{\nu}{d^2} t ,
$$
where $\bar{r}$ is the average radius
$$
\bar{r} = \frac{r_o + r_i}{2} = \frac{1+ \eta}{2} r_o.
$$
We rescale the velocity and the pressure by
\begin{align*}
    u(t,r,\varphi,z) = \frac{\nu}{d} \hat{u}(\hat{t}, \hat{x}, \hat{y},\hat{z}) ,\\
    p(t,r,\varphi,z) =\frac{\nu^2}{d^2} \hat{p}(\hat{t}, \hat{x}, \hat{y},\hat{z}) .
\end{align*}
Using the identities obtained in Appendix \ref{App1}, we have 
\begin{align*}
   \Big( \frac{\partial}{\partial \hat{t}} - \hat\Delta \Big) \hat{u}_{\hat{x}} 
   &+ \frac{\partial \hat p}{\partial \hat{x}} 
   = \mathfrak{R}\hat{x} \frac{\partial \hat{u}_{\hat{x}}}{\partial \hat{y}} 
   + \hat{\omega}(1 + \mu) \hat{u}_{\hat{y}} - (\hat{u} \cdot \nabla) \hat{u}_{\hat{x}} 
   +\cO \Bigl(\hat{\omega} (1 - \mu)\frac{\partial \hat{u}_{\hat{x}}}{\partial \hat{y}} \Bigr)
  \\
   &+ \cO(1 - \eta)^2 \hat{u}_{\hat{x}} - 2\Big(\hat{\omega} (1 - \mu)  \hat{x} + \cO(1 - \eta) \Big)\hat{u}_{\hat{y}} 
   +  \cO(1 - \eta) \Big( \hat{u}_{\hat{y}}^2 + \frac{\partial \hat{u}_{\hat{y}}}{ \partial \hat{y}} \Bigr),
 \\
   \Big( \frac{\partial}{\partial \hat{t}} - \hat\Delta \Big) \hat{u}_{\hat{y}} 
   &+ \frac{\partial \hat p}{\partial \hat{y}} 
= \mathfrak{R}\hat{x}\frac{\partial \hat{u}_{\hat{y}}}{\partial \hat{y}} + 
\hat{\omega} \frac{(1 - \mu)(1 + \eta^2)}{1 - \eta^2} \hat{u}_{\hat{x}} - \hat{\omega}(1 + \mu) \hat{u}_{\hat{x}} - (\hat{u} \cdot \nabla) \hat{u}_{\hat{y}}
  \\
   & 
+  \cO(1 - \eta) \Big( \hat{u}_{\hat{x}} \hat{u}_{\hat{y}} + \frac{\partial \hat{u}_{\hat{x}}}{ \partial \hat{y}} \Bigr) 
+ \cO \Bigl(\hat{\omega} (1 - \mu)\frac{\partial \hat{u}_{\hat{y}}}{\partial \hat{y}} \Bigr)+ \cO(1 - \eta)^2 \hat{u}_{\hat{y}},
\\
   \Big( \frac{\partial}{\partial \hat{t}} - \hat\Delta \Big) \hat{u}_{\hat{z}} 
   &+ \frac{\partial \hat p}{\partial \hat{z}} 
   = \mathfrak{R}\hat{x}\frac{\partial \hat{u}_{\hat{z}}}{\partial \hat{y}}  - (\hat{u} \cdot \nabla) \hat{u}_{\hat{z}} 
   +\cO \Bigl( \hat{\omega} (1 - \mu)\frac{\partial \hat{u}_{\hat{z}}}{\partial \hat{y}} \Bigr) ,
\end{align*}
$$
    \frac{\partial \hat{u}_{\hat{x}}}{\partial \hat{x}} 
    +\frac{\partial \hat{u}_{\hat{y}}}{\partial \hat{y}} + \frac{\partial \hat{u}_{\hat{z}}}{\partial \hat{z}} 
    = \cO(1-\eta) \Bigl( \hat{u}_{\hat{x}} + \frac{\partial \hat{u}_{\hat{y}}}{ \partial \hat{y}} \Bigr).\nonumber
$$
In this system, $\cO(\cdot)$ denotes various smooth functions that go to zero in the limit \eqref{limit}.
We drop the hats and the linearized system takes the form 
\begin{equation} \label{initial}
\partial_t u = {\cal L}_{(\eta,\mu)} u.
\end{equation}

%%%%%%%%%%%%%%%%%%%%%%%%%%%%%%%%%%%%%%%%%%%%%%%%%%%%%%%%%%%%%%

\section{Linear analysis in the axi-symmetric case}

%%%%%%%%%%%%%%%%%%%%%%%%%%%%%%%%%%%%%%%%%%%%%%%%%%%%%%%%%%%%%%

We first focus on radial, namely axi-symmetric instabilities, which are thus independent of $y$.

%%%%%%%%%%%%%%%%%%%%%%%%%%%%%%%%%%%%%%%%%%%%

\subsection{Study of the linearized operator $\mathcal{L}_{(\protect\eta ,\protect\mu )}$}

%%%%%%%%%%%%%%%%%%%%%%%%%%%%%%%%%%%%%%%%%%%%

For $y$ independent functions, the linearized axisymmetric Navier-Stokes equations (\ref{initial}) read
$$   \Big( \frac{\partial}{\partial  {t}} -  \Delta \Big)  {u}_{ {x}} 
   + \frac{\partial   p}{\partial  {x}} 
   = 
     {\omega}(1 + \mu)  {u}_{ {y}}   - 2\Big( {\omega} (1 - \mu)   ({x} + \cO(1 - \eta)) \Big) {u}_{ {y}} 
  + \cO(1 - \eta)^2  {u}_{ {x}}, 
  $$
  $$
   \Big( \frac{\partial}{\partial  {t}} -  \Delta \Big)  {u}_{ {y}} 
  = \Bigl[  {\omega} \frac{(1 - \mu)(1 + \eta^2)}{1 - \eta^2} 
 -  {\omega}(1 + \mu) \Bigr]  {u}_{ {x}} 
+ \cO(1 - \eta)^2  {u}_{ {y}},
   $$
   $$
   \Big( \frac{\partial}{\partial  {t}} -  \Delta \Big)  {u}_{ {z}} 
   + \frac{\partial   p}{\partial  {z}} 
   =  0,
   $$
   $$
    \frac{\partial  {u}_{ {x}}}{\partial  {x}} 
     + \frac{\partial  {u}_{ {z}}}{\partial  {z}} 
    = \cO(1-\eta)  {u}_{ {x}} .
$$
Let us take the Fourier transform in the $z$
variable and denote by $\alpha $ the Fourier variable. We moreover take the
Laplace transform in time and denote by $\lambda $ the Laplace variable.
Rescaling $u_{y}$ with 
\[
u_{y}=\mathfrak{R}\widehat{u}_{y},
\]%
and assuming that $\omega \mathfrak{R}$ is bounded, this leads to the system
(where $D=\partial /\partial x)$%
\begin{eqnarray*}
(\lambda +\alpha ^{2}-D^{2})u_{x}+Dp-T\widehat{u}_{y} &=&\mathcal{O}(1-\mu )%
\widehat{u}_{y}+\mathcal{O}(1-\eta )^{2}u_{x}, \\
(\lambda +\alpha ^{2}-D^{2})\widehat{u}_{y}-u_{x} &=&\mathcal{O} \Bigl(\frac{%
\omega }{\mathfrak{R}} \Bigr) u_{x}+\mathcal{O}(1-\eta )^{2}\widehat{u}_{y}, \\
(\lambda +\alpha ^{2}-D^{2})u_{z}+i\alpha p &=&0, \\
Du_{x}+i\alpha u_{z} &=&\mathcal{O}(1-\eta )u_{x}.
\end{eqnarray*}%
Now, we multiply the first equation by $\alpha ^{2}$ and add $i\alpha D$ times
the third one, which gives
\[
(\lambda +\alpha ^{2}-D^{2})(\alpha ^{2}u_{x}+i\alpha Du_{z})-\alpha ^{2}T%
\widehat{u}_{y}=\mathcal{O}(1-\mu )\widehat{u}_{y}+\mathcal{O}(1-\eta
)^{2}u_{x}.
\]
We replace $i\alpha Du_{z}$ by its expression from the fourth equation,
and obtain
\begin{align*}
(\lambda +\alpha ^{2}-D^{2})[(\alpha ^{2}-D^{2})u_{x}+\mathcal{O}(1-\eta
)Du_{x}]-\alpha ^{2}T\widehat{u}_{y} &=\mathcal{O}(1-\mu )\widehat{u}_{y}+%
\mathcal{O}(1-\eta )^{2}u_{x}, \\
(\lambda +\alpha ^{2}-D^{2})\widehat{u}_{y}-u_{x} &=\mathcal{O}(\frac{%
\omega }{\mathfrak{R}})u_{x}+\mathcal{O}(1-\eta )^{2}\widehat{u}_{y}.
\end{align*}%
We define now%
\[
\Gamma _{x}=\left[ (\alpha ^{2}-D^{2})+\mathcal{O}(1-\eta )D\right] u_{x},
\]
then the system above becomes
\[
(\lambda +\alpha ^{2}-D^{2})\Gamma _{x}-\alpha ^{2}T\widehat{u}_{y} =
\mathcal{O}(1-\mu )\widehat{u}_{y}+\mathcal{O}(1-\eta )^{2}\left[ (\alpha
^{2}-D^{2})+\mathcal{O}(1-\eta )D\right] ^{-1}\Gamma _{x},
\]
\begin{equation*}
(\lambda +\alpha ^{2}-D^{2})\widehat{u}_{y}-\left[ (\alpha ^{2}-D^{2})+
\mathcal{O}(1-\eta )D\right] ^{-1}\Gamma _{x} =\mathcal{O}(1-\eta )^{2}\widehat{u}_{y}
+\mathcal{O} \Bigl(\frac{\omega }{\mathfrak{R}} \Bigr)u_x.
\end{equation*}
Since the operator $\left[ (\alpha ^{2}-D^{2})+\mathcal{O}(1-\eta )D\right]^{-1}$ is bounded, taking into account of the boundary conditions $u_x(\pm1/2)=0$ 
 on $u_{x},$ it is clear that, for $(1-\eta ),(1-\mu )$ and $\omega / \mathfrak{R}$ close to $0,$ adding the boundary conditions on $u_{y}$ and the complementary ones on $u_{x}$ 
the operator on the right hand side is a
small relatively compact perturbation of the operator acting on the left
side on $(\Gamma _{x},\widehat{u}_{y}).$ 

This implies that the eigenvalues $%
\lambda $ are obtained via the simple perturbation theory from the
eigenvalues of the linear operator obtained for $\eta =1,\mu =1,$ i.e.
\begin{eqnarray}
(\lambda +\alpha ^{2}-D^{2})(\alpha ^{2}-D^{2})u_{x}-\alpha ^{2}T\widehat{u}%
_{y} &=&0,  \label{eigenvalusyst} \\
(\lambda +\alpha ^{2}-D^{2})\widehat{u}_{y}-u_{x} &=&0.  \nonumber
\end{eqnarray}
%%%%%%%%%%%%%%%
\subsection{Study of $\mathcal{L}_{1,\alpha}$}
%%%%%%%%%%%%%%%%%

For $(\eta ,\mu )=(1,1)$ we consider the eigenvalue problem
\begin{eqnarray}
\lambda u_{x} &=&(D^{2}-\alpha ^{2})u_{x}-Dp+T\widehat{u}_{y} \nonumber\\
\lambda \widehat{u}_{y} &=&(D^{2}-\alpha ^{2})\widehat{u}_{y}+u_{x} \nonumber\\
\lambda u_{z} &=&(D^{2}-\alpha ^{2})u_{z}-i\alpha p \label{basic lin syst}\\
0 &=&Du_{x}+i\alpha u_{z},\nonumber
\end{eqnarray}%
with the boundary conditions $u_{x}=\widehat{u}_{y}=u_{z}=0$ in $x=\pm 1/2.$ We denote the right-hand side of \eqref{basic lin syst}$_{1,2,3}$ to be the operator $\mathcal{L}_{1,\alpha}$ acting on $u = (u_{x},\widehat{u}_{y},u_{z})$. As
seen above, this leads to the 6th-order differential equation
\begin{equation} \label{order6}
(\lambda +\alpha ^{2}-D^{2})^{2}(\alpha ^{2}-D^{2})u_{x}-\alpha ^{2}Tu_{x}=0,
\end{equation}
with the boundary conditions%
\[
u_{x}=Du_{x}=\widehat{u}_{y}=0,x=\pm 1/2.
\]
Then we show the following
\begin{lemma}\label{lem3.1}
Let $\alpha > 0$. The spectrum of ${\cal L}_{1,\alpha}$ is only composed of real isolated (at least double) eigenvalues  
\begin{equation*}
    ... < \lambda_n<....<\lambda_2<\lambda_1<\lambda_0
\end{equation*}
where $\lambda_0<0$ for $T \le \alpha^4$, and $\lambda_n \rightarrow -\infty$ as $n \rightarrow \infty$.
\end{lemma}

\begin{proof}
The linear operator $\mathcal{L}_{1,\alpha }$ is symmetric with the following
suitable scalar product in $[L^{2}(-1/2,1/2)]^{3}:$%
\[
\langle u,v\rangle =\int_{-1/2}^{1/2}
u_{x}\overline{v_{x}}+T\hat{u}_{y}\overline{\hat{v}_{y}}+u_{z}\overline{v_{z}} \, dx.
\]
Indeed we obtain, using the divergence free condition on $u$ and $v,$ and integration by parts
\[
\langle \mathcal{L}_{1,\alpha }u,v\rangle =\langle u,\mathcal{L}_{1,\alpha
}v\rangle .
\]%
Since it is known that the Stokes operator on 
$$
\Bigl\{L^{2}[(-1/2,1/2)\times \mathbb{R} /(2\pi /\alpha ) \mathbb{Z}] \Bigr\}^{2}
$$ 
with the above boundary conditions
is self-adjoint, negative, and
idem for the operator acting on $u_{y}$, we deduce that the operator $%
\mathcal{L}_{1,\alpha }$ is self-adjoint, hence its spectrum is real. Moreover, as for the Stokes
operator, it has a compact resolvent, hence all its eigenvalues are real,
discrete, with only an accumulation at $-\infty .$ They all are at least
double because of the symmetry $\pm \alpha$.

Moreover, the largest eigenvalue $\lambda _{0}$
is negative if $T\leq\alpha ^{4}$ as can be seen by 
studying $\langle \mathcal{L}_{1,\alpha }u,u\rangle$.
\end{proof}

We now study when $0$ is an eigenvalue of ${\cal L}_{1,\alpha}$.
\begin{lemma}\label{lem3.2}
Let $\alpha > 0$,
$$
\sigma = (T\alpha ^{2})^{1/3},
$$
and let
\begin{eqnarray}
b_{1}^{2} &=&\sigma -\alpha ^{2}>0 
\nonumber \\
a_{2} &=&\frac{1}{\sqrt{2}}\sqrt{\alpha ^{2}+\frac{\sigma }{2}+\sqrt{\alpha
^{4}+\alpha ^{2}\sigma +\sigma ^{2}}},  \label{lambda_j} \\
b_{2} &=&-\frac{1}{\sqrt{2}}\sqrt{\sqrt{\alpha ^{4}+\alpha ^{2}\sigma +\sigma
^{2}}-\alpha ^{2}-\frac{\sigma }{2}}.  \nonumber
\end{eqnarray}%
Then $\lambda = 0$ is an eigenvalue of ${\cal L}_{1,\alpha}$ if and only if
$\mathfrak{F}_1 = 0$ or $\mathfrak{F}_2 = 0$ where
\begin{equation}
\mathfrak{F}_1:=-b_{1}\tan \frac{b_{1}}{2}(\cosh a_{2}+\cos b_{2})+(\sqrt{3}%
b_{2}-a_{2})\sinh a_{2}+(\sqrt{3}a_{2}+b_{2})\sin b_{2}  \label{delta+}
\end{equation}%
and
\begin{equation}
\mathfrak{F}_2:=b_{1}(\cosh a_{2}-\cos b_{2})+\tan \frac{b_{1}}{2}[(\sqrt{3}%
b_{2}-a_{2})\sinh a_{2}-(\sqrt{3}a_{2}+b_{2})\sin b_{2}] .  \label{delta -}
\end{equation}
The solutions of $\mathfrak{F}_1 =0$ give components $u_x, u_y$ even in $x$, and $u_z$ odd in $x$. The solutions of $\mathfrak{F}_2 =0$ give components $u_x, u_y$ odd in $x$ and $u_z$ even in $x$.
\end{lemma}

The proof of this Lemma is in Appendix \ref{App2}.
We now study the zeroes of $\mathfrak{F}_1$ and $\mathfrak{F}_2$.

\begin{lemma}
    Let $\alpha > 0$. Then $\mathfrak{F}_1(\alpha,T) = 0$ and $\mathfrak{F}_2(\alpha,T)=0$
    have infinitely many positive solutions $T$.
\end{lemma}

\begin{proof}
Note that, if we define $\tilde \sigma = \sigma / \alpha^2$, $a_2$ may be rewritten
$$
a_{2} = \frac{|\alpha|}{\sqrt{2}}  \sqrt{1+\frac{\tilde \sigma }{2}+\sqrt{1+ \tilde \sigma + \tilde \sigma^2}},
$$
and similarly for $b_2$.
Moreover, we have for large $\tilde \sigma$
$$
b_1 \sim \sqrt{\sigma}, \quad a_2 \sim \frac{\sqrt{3\sigma}}{2}, \quad b_2 \sim - \frac{\sigma}{2}.
$$
The end of the proof is then straightforward from expressions \eqref{delta+} and \eqref{delta -}.
\end{proof}

\begin{definition}
    For $\alpha > 0$, we define $T_c(\alpha)$ to be the smallest positive zero of $\mathfrak{F}_1$ and $\mathfrak{F}_2$.
\end{definition}

\begin{remark}\label{rem even}
    It is very easy to compute numerically $T_c(\alpha)$.
    Numerically, we find that $\mathfrak{F}_1 =0$ gives $T_c(\alpha)$ minimum for $\alpha = \alpha_c \approx 3.117$ 
    and $T_c(\alpha_c) \approx 1708$ (see Appendix \ref{Appendix22}), 
    which are the classical values \cite{Drazin}. The equation $\mathfrak{F}_2 =0$ gives a larger value for $T_c$. This is analogue to what happens (see \cite{Pellew}) in the convection problem, where the components $u_x,u_y$ of the eigenvector which are even in $x$ are more unstable that the odd ones.

\end{remark}

%%%%%%%%%%%%%%%%%%%%%%%%%%%%%

\subsection{A first bifurcation}

%%%%%%%%%%%%%%%%%%%%%%%%%%%%%

When $T < T_c$, all the eigenvalues of $\mathcal{L}_{1,\alpha}$ are
real and negative.
When $T = T_c$,  one eigenvalue is $0$.
For $T > T_c$, close to $T_c$,  one eigenvalue is positive and small.
As a consequence, taking into account of the $O(2)$ symmetry, we have a pitchfork bifurcation (see \cite{Iooss}) at $T_c$.

The eigenvalues and eigenvectors of the full system are perturbations 
of the eigenvalues and eigenvectors of $\mathcal{L}_{1,\alpha}$. 
They are of the form
$$
\zeta(x,z) = e^{\pm i \alpha z} \left( \begin{array}{c}
O(\mathfrak{R}^{-1}) \cr O(1) \cr O(\mathfrak{R}^{-1}) \end{array} \right).
$$
Let $\alpha = \alpha_c$. The eigenmodes corresponding to $0$ are of the form
\begin{equation*}
\zeta(x,z) = e^{ i\alpha_c z}\widehat{v}(x)
\end{equation*}%
with its complex conjugate for $-\alpha_c$. The bifurcating flow is called the ``Taylor vortex flow"  (TVF)
and has the form (we choose a solution invariant under the symmetry $z\rightarrow
-z)$:%
\begin{equation*}
A(\zeta +\overline{\zeta })+\mathcal{O}(|A|^{2}),\text{ }A\in 
\mathbb{R}.
\end{equation*}%
The Landau equation of $A$ is described in Appendix \ref{coefLandauequ} (see in particular formula
(\ref{Landauequ})).
We have a line of solutions in translating the $z$ coordinate
(changing the phase of amplitude $A)$. On this line there is another
symmetric solution obtained by shifting $z$ to $z+\pi /\alpha_c .$

%%%%%%%%%%%%%%%%%%%%%%%%%%%%%%%%%%%%%%%%%%%%%%%%%%%%%%%%%%%%%%

\section{Linear analysis in the non axisymmetric case}

%%%%%%%%%%%%%%%%%%%%%%%%%%%%%%%%%%%%%%%%%%%%%%%%%%%%%%%%%%%%%%

We now study general, non axi-symmetric perturbations.
 
%%%%%%%%%%%%%%%%%%%%%%%

\subsection{Set up}

%%%%%%%%%%%%%%%%%%%%%%%

At the leading order of the linearized system \eqref{initial}, we obtain the following system
\begin{align*}
   \Big( \frac{\partial}{\partial  {t}} -  \Delta \Big)  {u}_{ {x}} + \frac{\partial   p}{\partial  {x}} 
   &= {\mathfrak R} x \frac{\partial  {u}_{ {x}}}{\partial  {y}} 
   + 2 \omega  {u}_{ {y}} , \\
   \Big( \frac{\partial}{\partial  {t}} -  \Delta \Big)  {u}_{ {y}} + \frac{\partial   p}{\partial  {y}} 
   &= {\mathfrak R} x \frac{\partial  {u}_{ {y}}}{\partial  {y}} + \mathfrak{R} {u}_{ {x}},\\
   \Big( \frac{\partial}{\partial  {t}} -  \Delta \Big)  {u}_{ {z}} + \frac{\partial   p}{\partial  {z}} 
   &=  {\mathfrak R} x  \frac{\partial  {u}_{ {z}}}{\partial  {y}} ,\\
    \frac{\partial  {u}_{ {x}}}{\partial  {x}} +\frac{\partial  {u}_{ {y}}}{\partial  {y}} 
    + \frac{\partial  {u}_{ {z}}}{\partial  {z}} 
    &= 0.
\end{align*}
We take the Fourier transform in $z$ with dual variable $\alpha$, the Fourier transform in $y$, with dual 
variable $\beta$, and the Laplace transform in time, with dual variable $\lambda$. This leads to
\begin{align*}
    (\lambda - \partial_x^2 + \alpha^2 + \beta^2) u_x + \frac{\partial p}{\partial x} &= i \beta \mathfrak{R} x u_x + 2 \omega u_y,\\
(\lambda - \partial_x^2 + \alpha^2 + \beta^2) u_y + i \beta p &= i \beta \mathfrak{R} x  u_y + \mathfrak{R} u_x,\\
(\lambda - \partial_x^2 + \alpha^2 + \beta^2) u_z + i \alpha p &= i \beta \mathfrak{R} x  u_z,\\
\frac{\partial u}{\partial x} + i \beta u_y + i \alpha u_z &= 0.
\end{align*}
It is thus natural to assume that 
$$
\mathfrak{B} = \beta \mathfrak{R}
$$ 
remains constant as $\mathfrak{R}$ goes to infinity.
As a consequence, $\beta$ goes to $0$ as $\mathfrak{R}$ goes to infinity. The system thus simplifies into
\begin{equation}\label{eqn:nonaxi}
    \begin{aligned}
    (\lambda +\alpha ^{2}-D^{2}-i\mathfrak{B}x)u_{x}+Dp &=2\omega u_{y}, \\
(\lambda +\alpha ^{2}-D^{2}-i\mathfrak{B}x)u_{y}+i\beta p &=\mathfrak{R%
}u_{x}, \\
(\lambda +\alpha ^{2}-D^{2}-i\mathfrak{B}x)u_{z}+i\alpha p &=0, \\
Du_{x}+i\beta u_{y}+i\alpha u_{z} &=0,
\end{aligned}
\end{equation}
where $D = \partial_x$. 
Applying $(\lambda +\alpha ^{2}-D^{2}-i\mathfrak{B}x)$ to \eqref{eqn:nonaxi}$_4$ 
and using \eqref{eqn:nonaxi}$_{2,3}$ leads to%
\begin{equation}\label{eqn:nonaxi_u}
    (\lambda +\alpha ^{2}-D^{2}-i\mathfrak{B}x)Du_{x}+(\alpha ^{2}+\beta ^{2})p=-i\mathfrak{B}u_{x}.
\end{equation}
Hence, $\beta ^{2}$ should be neglected with respect to $\alpha ^{2}.$ Applying $-D$ to \eqref{eqn:nonaxi_u} and adding $\alpha^2$ times \eqref{eqn:nonaxi}$_1$, we have
\[
(\lambda +\alpha ^{2}-D^{2}-i\mathfrak{B}x)(\alpha ^{2}-D^{2})u_{x}=2\omega \alpha ^{2}u_{y}.
\]
Finally, rescaling $u_y$ in  
\[
u_{y}=\mathfrak{R}\widehat{u}_{y},
\] 
it is now clear that $i\beta p$ should be neglected in the second
line of (\ref{eqn:nonaxi}).
We obtain the system%
\begin{eqnarray} \label{non1}
(\lambda +\alpha ^{2}-D^{2}-i\mathfrak{B}x)(\alpha ^{2}-D^{2})u_{x}
 &=&T\alpha ^{2}\widehat{u}_{y}\\
(\lambda +\alpha ^{2}-D^{2}-i\mathfrak{B}x)\widehat{u}_{y} &=&u_{x}, \label{non2}
\end{eqnarray}%
where the parameters are $\alpha ^{2},T,\mathfrak{B}$, and where the boundary
conditions are%
\[
u_{x}=\widehat{u}_{y}=Du_{x}=0,\quad x=\pm 1/2.
\]
In the sequel, we will drop the hat on $u_y$.

%%%%%%%%%%%%%%%%%%

\subsection{Expansion of $\lambda$}\label{explambda}

%%%%%%%%%%%%%%%%%%

In the previous section we have rescaled $u_y$ in $u_{y}=\mathfrak{R}\widehat{u}_{y},$
and have neglected $\beta^2$ with respect to $\alpha^2$ and $\beta p$ in the equation on $u_y$.
The linear system which we study (to be compared with (\ref{initial}))  is now
\begin{eqnarray}
(\partial _{t}-\Delta _{\bot })u_{\bot }+\nabla _{\bot }p &=&\mathfrak{R}%
x\partial _{y}u_{\bot }+(T\widehat{u}_{y},0)^{t},  \label{newlinear} \\
(\partial _{t}-\Delta _{\bot })\widehat{u}_{y} &=&\mathfrak{R}x\partial _{y}%
\widehat{u}_{y}+u_{x},  \nonumber \\
\nabla_{\bot} \cdot u_{\bot }+\mathfrak{R}\partial _{y}\widehat{u}_{y} &=&0,  \nonumber
\end{eqnarray}%
with $2\pi /\alpha $ periodicity in $z$ and boundary conditions%
\[
u_{\bot }=\widehat{u}_{y}=0,x=\pm 1/2,
\]%
and where the index $\bot $ means $(x,z)$ components. Looking for solutions
of the form%
\[
u=e^{i(\alpha z+\beta y)+\lambda t}(u_{\bot },\widehat{u}_{y})^{t},
\]%
where $u_{\bot },\widehat{u}_{y}$ are functions of $x$,
we obtain the system (\ref{non1},\ref{non2}) in setting $\mathfrak{B}=\beta \mathfrak{R}$,
 $T=2\omega \mathfrak{R}$.

We know that $\lambda =0$ is the largest eigenvalue when $\mathfrak{B}=0,$
and $T=T_c(\alpha, 0 ).$ Moreover, $T_c(\alpha,0)$ has a strict minimum at $\alpha _{c}$ (see figure \ref{axi} for numerical evidence), i.e.
\[
\frac{dT_c}{d\alpha }(\alpha _{c}, 0)=0,
\qquad 
\frac{d^{2}T_c}{d\alpha ^{2}}(\alpha_c,0)>0.
\]%
We note that, for a given $(\alpha,\mathfrak{B})$,
the eigenvalue  $\lambda$ gives a unique eigenvector $U=(u_{x},u_{y})(x)$, 
even though $\lambda$ is a double eigenvalue when we take into account of the symmetry $z \rightarrow -z$ and of the third component $u_{z}$. 
The analyticity of the linear operator in $(\alpha,\mathfrak{B})$ leads to the same analyticity for $\lambda$.

We first prove that the eigenvalues $\lambda$ are real.

\begin{lemma}\label{lem4.1}
The eigenvalues $\lambda$ are real and are even functions of $\alpha$ and of $\mathfrak{B}$.
\end{lemma}
\begin{proof}
Thanks to the form of system (\ref{non1},\ref{non2}) 
we have the following symmetry properties for $\lambda(\alpha,\mathfrak{B})$:

\begin{itemize}

\item[i)] changing $\mathfrak{B}$ in $-\mathfrak{B}$ changes $%
\lambda $ in $\overline{\lambda },$ and for $%
\mathfrak{B}=0,$ $\lambda $ is real.

\item[ii)] Changing $(x,\mathfrak{B},u_{x}(x),u_{y}(x))$ into 
$(-x,-\mathfrak{B},u_{x}(-x),u_{y}(-x))$ shows that the same $\lambda$ is valid for the eigenvector $(u_{x}(-x),u_{y}(-x))$. Hence 
\[
\lambda(\alpha,\mathfrak{B})=\lambda(\alpha,-\mathfrak{B})=
\overline{\lambda}(\alpha,\mathfrak{B}).
\]
 \end{itemize}
Finally, the evenness property in $\alpha$ is obvious.   
\end{proof}

Let us give the Taylor expansion of the largest eigenvalue $\lambda_0 $ in the
neighborhood of $0,$ for $(\alpha ,T,\mathfrak{B})$ close to $(\alpha
_{c},T_{c},0)$. From the previous Lemma, we have
\begin{equation}\label{lambda_taylor}
    \begin{aligned}
\lambda_0 &= a_{2}(\alpha ^{2}-\alpha _{c}^{2})+a_{3}\tau
+a_{4} \mathfrak{B}^{2}
+a_{6}(\alpha^{2}-\alpha _{c}^{2})^{2} +a_7\mathfrak{B}^2(\alpha^{2}-\alpha _{c}^{2}) 
\\
& \quad + a_{9}\tau(\alpha^{2}-\alpha _{c}^{2})+a_{10}\rho^2 + \cdots,
\end{aligned}
\end{equation}
where all coefficients $a_{j}$ are real and
\[
\tau :=T-T_{c}.
\]
Note that $\alpha^2 - \alpha_c^2$ is at leading order proportional to $\alpha - \alpha_c$ since
$$
\alpha^2 - \alpha_c^2 = (\alpha - \alpha_c) (\alpha + \alpha_c)
= 2 \alpha_c (\alpha - \alpha_c) + \cO \Bigl( (\alpha - \alpha_c)^2 \Bigr).
$$

\begin{lemma}\label{lem4.2}
    We have
    $ a_2= 0$, $a_3 > 0$, 
    and we obtain numerically that $a_4< 0$,  $a_6 < 0$ . 
\end{lemma}

\begin{proof} 
Since $\tau =0$ corresponds to
the strict minimum of $T$ at $\alpha =\alpha _{c}$ when $\fB = 0$,  we have
\begin{align}
    \frac{\partial \tau}{\partial \alpha}(\alpha_c , 0) &= 0, \label{tau_alpha_first}\\
    \frac{\partial^2 \tau}{\partial \alpha^2}(\alpha_c , 0) &> 0. \label{tau_alpha_second}
\end{align}
Fixing $\lambda_0 = \fB = 0$ in \eqref{lambda_taylor}, taking derivative with respect to $\alpha$ and evaluating at $\alpha =\alpha _{c},$ \eqref{tau_alpha_first} implies $a_2 = 0$.

Next we prove that the coefficient $a_{3}$ is positive. Let us denote by $%
L_{0}$ the linear operator defined as%
\[
L_{0}U=\left( 
\begin{array}{c}
-(\alpha _{c}^{2}-D^{2})^{2}u_{x}+T_{c}\alpha _{c}^{2}u_{y} \\ 
(D^{2}-\alpha _{c}^{2})u_{y}+u_{x}%
\end{array}%
\right) ,
\]%
where 
\[
U=(u_{x},u_{y})^{t},u_{x}=Du_{x}=u_{y}=0,x=\pm 1/2.
\]%
Now the eigenvector of $L_{0}$ for the eigenvalue $0$ is%
\[
\zeta =(u_{x}^{0},u_{y}^{0})^{t},
\]%
so that%
\[
L_{0}\zeta =0.
\]%
Moreover, 
it is seen in Lemma \ref{lem3.2} (see also Remark \ref{rem even}) that $u_{x}^{0},u_{y}^{0}$ are even in $x.$ The
system (\ref{non1}), (\ref{non2}) may be written as%
\begin{eqnarray}
\lambda_0 (H_{0}+\widehat{\alpha }H_{010})U &=&L_{0}U+\mathfrak{B}L_{100}U+%
\widehat{\alpha }L_{010}U+\tau L_{001}U+  \nonumber \\
&&+\mathfrak{B}\widehat{\alpha }L_{110}U+\widehat{\alpha }\tau L_{011}U+%
\widehat{\alpha }^{2}L_{020}U,  \label{eigensystem}
\end{eqnarray}%
with $\widehat{\alpha }=\alpha ^{2}-\alpha _{c}^{2}$ and%
\[
H_{0}U=\left( 
\begin{array}{c}
(\alpha _{c}^{2}-D^{2})u_{x} \\ 
u_{y}%
\end{array}%
\right) ,H_{010}U=\left( 
\begin{array}{c}
u_{x} \\ 
0%
\end{array}%
\right) ,
\]%
\[
L_{100}U=\left( 
\begin{array}{c}
ix(\alpha _{c}^{2}-D^{2})u_{x} \\ 
ixu_{y}%
\end{array}%
\right) ,L_{010}U=\left( 
\begin{array}{c}
-2(\alpha _{c}^{2}-D^{2})u_{x}+T_{c}u_{y} \\ 
-u_{y}%
\end{array}%
\right) ,
\]%
\[
L_{001}U=\left( 
\begin{array}{c}
\alpha _{c}^{2}u_{y} \\ 
0%
\end{array}%
\right) ,L_{110}U=\left( 
\begin{array}{c}
ixu_{x} \\ 
0%
\end{array}%
\right) ,L_{011}U=\left( 
\begin{array}{c}
u_{y} \\ 
0%
\end{array}%
\right) ,L_{020}=\left( 
\begin{array}{c}
-u_{x} \\ 
0%
\end{array}%
\right) .
\]%
As for the proof of selfadjointness of $\mathcal{L}_{1,\alpha },$ we may
prove that $L_{0}$ is selfadjoint in $[L^{2}(-1/2,1/2)]^{2}$ with the scalar
product%
\[
\langle U,V\rangle =\int_{-1/2}^{1/2}(u_{x}\overline{v_{x}}+\alpha
_{c}^{2}T_{c}u_{y}\overline{v_{y}})dx,
\]%
so that%
\[
\langle L_{0}U,V\rangle =\langle U,L_{0}V\rangle,\text{   for any }U,V\in dom(L_{0}).
\]%
It results that $\zeta $ is orthogonal to the range of $L_{0}.$ In the
following computation, we use%
\[
\langle H_{0}U,U\rangle =\int_{-1/2}^{1/2}(|Du_{x}|^{2}+\alpha
_{c}^{2}|u_{x}|^{2}+\alpha _{c}^{2}T_{c}|u_{y}|^{2})dx>0
\]%
for $U\neq 0.$ We also expand the eigenvector $U$ as%
\[
U=\zeta +\mathfrak{B}U^{(100)}+\widehat{\alpha }U^{(010)}+\tau U^{(001)}+%
\mathfrak{B}\widehat{\alpha }U^{(110)}+\widehat{\alpha }\tau U^{(011)}+%
\widehat{\alpha }^{2}U^{(020)}+\mathfrak{B}^{2}U^{(200)}+...
\]%
and identify the coefficients of $\widehat{\alpha },\tau ,\mathfrak{B,B}^{2}$
in (\ref{eigensystem}) where $\lambda $ is expanded in (\ref{lambda_taylor}):%
\begin{eqnarray*}
a_{2}H_{0}\zeta  &=&L_{0}U^{(010)}+L_{010}\zeta , \\
a_{3}H_{0}\zeta  &=&L_{0}U^{(001)}+L_{001}\zeta , \\
0 &=&L_{0}U^{(100)}+L_{100}\zeta , \\
a_{4}H_{0}\zeta  &=&L_{0}U^{(200)}+L_{100}U^{(100)}.
\end{eqnarray*}%
By construction of $\alpha _{c},T_{c}$ we have $a_{2}=0,$ hence%
\[
\langle L_{010}\zeta ,\zeta \rangle =0,
\]%
i.e.%
\begin{equation}
-\int_{-1/2}^{1/2}(2|Du_{x}^{0}|^{2}+2\alpha
_{c}^{2}|u_{x}^{0}|^{2}+\alpha
_{c}^{2}T_{c}|u_{y}^{0}|^{2})dx+%
\int_{-1/2}^{1/2}T_{c}u_{y}^{0}u_{x}^{0}dx=0  \label{a2=0}
\end{equation}%
We also obtain%
\[
a_{3}\langle H_{0}\zeta ,\zeta \rangle =\langle L_{001}\zeta ,\zeta \rangle
=\alpha _{c}^{2}\int_{-1/2}^{1/2}u_{y}^{0}u_{x}^{0}dx
\]%
hence, from (\ref{a2=0}) and $\langle H_{0}\zeta ,\zeta \rangle >0,$ we
deduce that%
\[
a_{3}>0.
\]%
The equation for $U^{(100)}$ is valid since $L_{100}\zeta $ has odd
components in $x,$ hence orthogonal to $\zeta $ which has even components.
Now we obtain%
\[
a_{4}\langle H_{0}\zeta ,\zeta \rangle =\langle L_{100}U^{(100)},\zeta
\rangle 
\]%
which needs to be computed numerically, at least to check that $a_{4}<0$. 

Fixing $\lambda_0 = \fB = 0$ in \eqref{lambda_taylor} again, and taking second derivative with respect to $\alpha$ and evaluating at $\alpha =\alpha _{c},$ we have
$$0 = a_3 \frac{\partial^2 \tau}{\partial \alpha^2}(\alpha_c , 0) + 8 \alpha_c^2 a_6.$$
Since $a_3 > 0$ and $\tau$ has a strict minimum in $\alpha_c$, this implies (numerically) $a_6 < 0$ by \eqref{tau_alpha_second}.

Finally, $a_4 < 0$ follows from the numerical evidence  (see Appendix \ref{Appendix22}) that for $\lambda_0 = 0$ and fixed $\alpha$, the critical $T$ is an increasing convex function with respect to $\fB > 0$.

In conclusion, the principal part of $\lambda $ becomes
\begin{equation}
\lambda_0 =a_{3}\tau +a_{4}\mathfrak{B}^{2}+a_{6}(\alpha^{2}-\alpha _{c}^{2})^{2}+a_{7}\mathfrak{B}^{2}(\alpha^{2}-\alpha _{c}^{2})
+ a_{9}\tau(\alpha^{2}-\alpha _{c}^{2})+a_{10}\tau^2
\label{equ dispersion}
\end{equation}
where $a_{j}$ are real. 
\end{proof}

%%%%%%%%%%%%%%%%%%%%%%

\subsection{Study of the critical Taylor number}

%%%%%%%%%%%%%%%%%%%%%%

Let us now study how $T_c$ depends on $\alpha$ and $\mathfrak{B}$.
Let
$$
b_4 = - a_4 >0, \qquad b_6 = - a_6 > 0
$$
where the strict signs of $a_4$ and $a_6$ come from numerical computation (see Appendices \ref{Appendix21}, and \ref{Appendix22}).
From the expression (\ref{equ dispersion}) of $\lambda$, we deduce that the critical value of $T$ for $\alpha\ne \alpha_{c}$ is
approximately reached for
\begin{equation}
\alpha^2-\alpha_{c}^2 \approx \Bigl[\frac{a_7 }{2b_6}+\frac{a_{9}b_{4}}{2b_{6}a_{3}} \Bigr]\mathfrak{B}^2
\end{equation}
and 
\[
T \approx T_{c}+\frac{b_4}{a_3}\mathfrak{B}^2- \frac{1}{4b_{6}a_{3}^3} \Bigl[(a_{3}a_{7}+a_{9}b_{4})^2+4a_{10}b_{6}a_{4}^2 \Bigr]\mathfrak{B}^4
 >T_{c}.
\]

%%%%%%%%%%%%%%%%%%%%%%%%%%%%%%%%%%%%%%%%%%%%%%%%%%%%%%%%%%%%%

\section{The Ginzburg Landau equation}

%%%%%%%%%%%%%%%%%%%%%%%%%%%%%%%%%%%%%%%%%%%%%%%%%%%%%%%%%%%%%

%%%%%%%%%%%%%%%%%%%%%%%%%%%%%%%%%%%%%%%%%%%%%%%%%%%%%%%%%%%%%
\subsection{Limit of Navier-Stokes system}
%%%%%%%%%%%%%%%%%%%%%%%%%%%%%%%%%%%%%%%%%%%%%%%%%%%%%%%%%%%%%

From now on we are dealing with the limit system
\begin{eqnarray}
(\partial _{t}-\Delta _{\bot })u_{\bot }+\nabla _{\bot }p &=&\mathfrak{R}%
x\partial _{y}u_{\bot }+(T\widehat{u}_{y},0)^{t}-(u_{\bot}\cdot\nabla_{\bot})u_{\bot},  \label{newnonlinear} \\
(\partial _{t}-\Delta _{\bot })\widehat{u}_{y} &=&\mathfrak{R}x\partial _{y}%
\widehat{u}_{y}+u_{x}-(u_{\bot}\cdot\nabla_{\bot})\widehat{u}_{y},  \nonumber \\
\nabla_{\bot} \cdot u_{\bot }+\mathfrak{R}\partial _{y}\widehat{u}_{y} &=&0,  \nonumber
\end{eqnarray}%
with $2\pi /\alpha_c $ periodicity in $z$ and boundary conditions%
\[
u_{\bot }=\widehat{u}_{y}=0,x=\pm 1/2.
\]%
Here the index $\bot $ means $(x,z)$ components. The system (\ref{newnonlinear}) is derived from the system in section \ref{sec2} by rescaling of $u_y$ and dropping the small terms such as $\mathfrak{R}^{-1} \partial_yp$, as well as terms with factors of order $(1-\eta)$ and $(1-\mu)$. Moreover, since we are looking for slowly varying functions of $y$, we also suppress $\partial_{y}^2$ in linear terms. Finally, because we seek solutions close to $0$, we suppress quadratic terms involving the operator $\mathfrak{R}\partial_y$, while keeping such terms in the linear part. The linear part of \eqref{newnonlinear} is studied in subsection \ref{explambda}.

%%%%%%%%%%%%%%%%%%%%%%%%%%%%%%%

\subsection{The amplitude equation}

%%%%%%%%%%%%%%%%%%%%%%%%%%%%%%%

Heuristically, the solutions of the limit Navier-Stokes equations (\ref{newnonlinear}) can be described by the solutions of the
following Ginzburg-Landau equation
\begin{equation}
\frac{\partial  A}{\partial t}=a_{3}\tau A+b_{4}\mathfrak{R}^2 \frac{\partial^2 A}{\partial y^2} -cA|A|^2, \label{GLequ}
\end{equation}
where the coefficients $a_3$, $a_4 = -b_4$ and $c$ are computed in Appendices \ref{coefLandauequ} and \ref{compGLequ}.

We do not attempt to justify this assertion in the general time-dependent case, and instead restrict our attention to steady solutions.
The mathematical justification for the search of steady solutions comes back to the spatial dynamics theory (see \cite{Iooss2}), where the coordinate $y$ plays the role of time. In this setup, we look for solutions (now independent of time) that are small and bounded for $y\in \mathbb{R}$. This would give a 4th-order ODE for $A(y)\in \mathbb{C}$, which coincides with the equation derived below for the principal part of steady solutions.

In particular, small-amplitude steady solutions of the initial Navier-Stokes system are
described by small-amplitude steady solutions of the Ginzburg-Landau equation.
We do not prove this assertion in this article, since it is a direct but lengthy application of the 
techniques developed in \cite{Iooss2}. Accordingly,
the search of time-independent solutions of the Navier-Stokes can thus be reduced to the search of
time-independent solutions of the Ginzburg-Landau equation, which is an easier task that we now detail.

We note that, for any $\theta \in \mathbb{R}$, 
$$
A(y) = \sqrt{a_3 \tau \over c} e^{i \theta}
$$
is a stationary solution of (\ref{GLequ}), which corresponds to the Taylor vortex flows.

There exists another family of solutions of (\ref{GLequ}) which are under the form
$$
\rho e^{i(\alpha_c z +\beta y)}+ c.c.,
$$
where
$$
\rho^2 =\frac{a_3}{c}\tau-\frac{b_4}{c}\beta^2\mathfrak{R}^2.
$$
These solutions corresponds to wavy vortices. 

The secondary bifurcation to such solutions occurs for (recall that $T=2\omega \mathfrak{R}$)
\[
\omega=\omega_c+\frac{b_4}{2a_3}\frac{\mathfrak{B}^2}{\mathfrak{R}}.
\]
This leads in particular to a secondary bifurcation branching from Couette flow (see Figure \ref{diagram}), under the form of steady wavy vortices 
(defined up to a shift in $z,y$).

The branches of steady solutions obtained for $A=\rho e ^{i\beta y}$ with $\beta\ne0$ do not meet each others, since they are parallel to the TVF branch (see Figure \ref{diagram}). Notice that these wavy vortices are steady in the rotating frame (rotation rate $\Omega_{rf}$), hence they would correspond to time periodic wavy vortices for an observer.
\begin{figure}[th]
\begin{center}
\includegraphics[width=4cm]{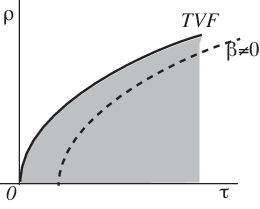}
\end{center}
\caption{TVF is given by $a_{3}\tau=c \rho^2$. Any point $(\tau,\rho)$ in the grey region corresponds to wavy vortices with $A=\rho e^{i\beta y}$ such that $b_{4}\mathfrak{R}^2 \beta^2=a_{3}\tau-c\rho^2$. These solutions correspond to the boundary of the bounded region described in Figure \ref{caseKnon0}.}
\label{diagram}
\end{figure}

%%%%%%%%%%%%%%%%%%%%%%%%%%%%%%%%%

\subsection{Stability of bifurcating steady wavy vortices}

%%%%%%%%%%%%%%%%%%%%%%%%%%%%%%%%%

Let us start with the Ginzburg-Landau equation (\ref{GLequ}) ruling the principal part of
bifurcating solutions 
and study the stability of the steady solution%
\begin{eqnarray}
A &=&A_{0}e^{i\beta _{0}y},  \label{wavy vortices} \\
0 &=&a_{3}\tau -\beta _{0}^{2}b_{4}\mathfrak{R}^{2}-c|A_{0}|^{2} \nonumber
\end{eqnarray}%
with respect to the time dependent Ginzburg-Landau equation.

Let us set%
\[
A=A_{0}e^{i\beta _{0}y}+A^{\prime },
\]%
and consider the linearized system for $(A^{\prime },\overline{A^{\prime }}):
$%
\[
\frac{\partial A^{\prime }}{\partial t}=a_{3}\tau A^{\prime }+b_{4}\mathfrak{%
R}^{2}\frac{\partial ^{2}A^{\prime }}{\partial y^{2}}-2c|A_{0}|^{2}A^{\prime
}-cA_{0}^{2}\overline{A^{\prime }}.
\]%
We only study the stability with respect to perturbations of the form $%
A^{\prime }=|A^{\prime }|e^{i\beta ^{\prime }y}$ where $\beta ^{\prime }$ is
different from $\beta _{0}$ in general. Then we have to discuss the
eigenvalues of%
\[
\left( 
\begin{array}{cc}
b_{4}\mathfrak{R}^{2}(\beta _{0}^{2}-\beta ^{\prime 2})-c|A_{0}|^{2} & 
-cA_{0}^{2} \\ 
-c\overline{A_{0}}^{2} & b_{4}\mathfrak{R}^{2}(\beta _{0}^{2}-\beta ^{\prime
2})-c|A_{0}|^{2}%
\end{array}%
\right) ,
\]%
where we used%
\[
a_{3}\tau -\beta ^{\prime 2}b_{4}\mathfrak{R}^{2}-2c|A_{0}|^{2}=b_{4}%
\mathfrak{R}^{2}(\beta _{0}^{2}-\beta ^{\prime 2})-c|A_{0}|^{2}.
\]%
The eigenvalues are%
\begin{eqnarray*}
&&b_{4}\mathfrak{R}^{2}(\beta _{0}^{2}-\beta ^{\prime 2}), \\
&&b_{4}\mathfrak{R}^{2}(\beta _{0}^{2}-\beta ^{\prime 2})-2c|A_{0}|^{2}.
\end{eqnarray*}%
It results that the steady solution (\ref{wavy vortices}) is stable for $%
|\beta ^{\prime }|>|\beta _{0}|,$ and unstable for $|\beta ^{\prime
}|<|\beta _{0}|.$

Notice that TVF corresponds to $\beta _{0}=0,$ so that it is stable with
respect to these perturbations, periodic in y, while for the wavy vortices
with $\beta _{0}\neq 0,$ they are unstable with respect to perturbations
such that $0\leq |\beta ^{\prime }|<|\beta |.$

We expect that the genuine solutions of the Navier-Stokes equations have the same stability properties
as the corresponding solutions of the Ginzburg-Landau equation.

%%%%%%%%%%%%%%%%%%%%%%%%%%%%%%

\subsection{Stationary solutions of the amplitude equation}

%%%%%%%%%%%%%%%%%%%%%%%%%%%%%%

We now study all the stationary solutions of the time-independent Ginzburg-Landau equation.
We recall that the principal part of a steady velocity vector field is%
\[
A(y)e^{i\alpha_c z}\widehat{U}(x)+\overline{A}(y)e^{-i\alpha_c z}\overline{%
\widehat{U}}(x)
\]%
where $A(y)\in \mathbb{C}$, $y\in \mathbb{R}$ 
and $\widehat{U}(x)$ is the eigenvector $(u_{x}^{0},u_{y}^{0},u_{z}^{0})(x)$
defined in previous section. Let us now look for solutions $A(y)$ of%
\[
a_{3}\tau A+b_{4}\mathfrak{R}^{2}\frac{\partial ^{2}A}{\partial y^{2}}%
-cA|A|^{2}=0,
\]%
which are small and bounded, when $\tau$ is small. We claim that
this (reversible) system is integrable.

First we change the scale in $y,$ and define the new positive coefficients%
\begin{equation}
\mathfrak{R}\widetilde{y} = y, \qquad
c_{0} = \frac{c}{2b_{4}}, \qquad a_{0}=\frac{a_{3}}{b_{4}}.
\end{equation}%
Then the equation reads (suppressing the tilde)%
\begin{equation}
\frac{\partial ^{2}A}{\partial y^{2}}=-a_{0}\tau A+2c_{0}A|A|^{2}.
\label{steadyamplitude}
\end{equation}%
Let us set%
\[
A(y)=\rho(y) e^{i\theta(y)},
\]%
then%
\begin{eqnarray*}
\rho ^{\prime \prime }-\rho \theta ^{\prime 2} &=&2c_{0}\rho ^{3}-a_{0}\tau
\rho , \\
2\rho ^{\prime }\theta ^{\prime }+\rho \theta ^{\prime \prime } &=&0,
\end{eqnarray*}%
hence%
\[
\rho ^{2}\theta ^{\prime }=K,
\]%
\[
\rho ^{\prime \prime }=\frac{K^{2}}{\rho ^{3}}+2c_{0}\rho ^{3}-a_{0}\tau \rho,
\]%
where $K$ is some constant.
Now defining%
\[
u(\rho )=\rho ^{\prime 2},
\]%
we finally obtain%
\[
u^{\prime }(\rho )=\frac{2K^{2}}{\rho ^{3}}+4c_{0}\rho ^{3}-2a_{0}\tau \rho 
\]%
hence%
\[
u(\rho )=-\frac{K^{2}}{\rho ^{2}}+c_{0}\rho ^{4}-a_{0}\tau \rho ^{2}+H,
\]%
where $H$ is some other constant. It remains to look for bounded solutions of the
first order differential equation%
\[
\rho ^{\prime 2}=-\frac{K^{2}}{\rho ^{2}}+c_{0}\rho ^{4}-a_{0}\tau \rho
^{2}+H.
\]%
The cases with $\theta ^{\prime }=\beta \mathfrak{R}=const$ \ and $\rho
=const$ give%
\[
-\frac{K^{2}}{\rho ^{2}}+c_{0}\rho ^{4}-a_{0}\tau \rho ^{2}+H=0
\]%
and correspond to periodic solutions $A(y)=\rho e^{i\beta y}.$ 

There are many other solutions, and we are interested in solutions such that as $y$
tends to $\pm \infty ,$ $A(y)$ goes to the same bounded limit (small for $%
\mu $ small) in such a way that this will respect the azimuthal periodicity,
since $y$ replaces $\varphi .$

Let us first study the case $K=0,$ where we may take $A$ real. The system (%
\ref{steadyamplitude}) leads to%
\[
A^{\prime 2}=-a_{0}\tau A^{2}+c_{0}A^{4}+H,
\]%
and we can describe the phase portrait of the solutions in the plane $%
(A,A^{\prime })$ as indicated in Figure \ref{caseK=0}. For 
\[
0<H<\frac{(a_{0}\tau )^{2}}{4c_{0}},
\]%
there is a family of closed orbits, surrounding $0$ corresponding to
periodic solutions $\rho (y),$ of period running from $\frac{2\pi }{\sqrt{%
a_{0}\tau }}$ (multiplied by $\mathfrak{R}$ in the original scaling), to $%
\infty .$ All these solutions are eligible for our problem. The limit case
where $H=\frac{(a_{0}\tau )^{2}}{4c_{0}}$ leads to 
\[
A^{\prime 2}=c_{0} \Bigl( A^{2}-\frac{a_{0}\tau }{2c_{0}} \Bigr)^{2}
\]%
which gives a heteroclinic orbit connecting the equilibrium points $A=\pm 
\sqrt{\frac{a_{0}\tau }{2c_{0}}}$. 
These points correspond to the Taylor Vortex Flow solution  (TVF).
In this case, the phase $\theta$ is constant, and at each infinity the limits for $A$ are different, corresponding to a vertical shift of TVF by $\pi/\alpha_c$. We may consider the juxtaposition along the $y$ coordinate of this solution with the symmetric one, then recovering the periodicity in the azimuthal coordinate.

Very close to this heteroclinic we have periodic solutions with very long periods, staying  very close to TVF flow for long intervals in $y$. These solutions are clearly eligible.

\begin{figure}[th]
\begin{center}
\includegraphics[width=6cm]{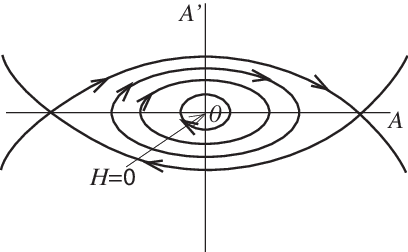}
\end{center}
\caption{Bounded solutions for K=0}
\label{caseK=0}
\end{figure}

Let us now consider the case where $K\neq 0.$ Let us make a new scaling to simplify the computations
\begin{equation}
\rho   = k\widetilde{\rho }, \quad K=k\widetilde{K},
\quad c_{0}k^{4}=1, \quad a_{0}\tau k^{2} =\widetilde{\tau }.
\end{equation}%
Then, suppressing the tildes, we have 
\[
\rho ^{\prime 2}=-\frac{K^{2}}{\rho ^{2}}+\rho ^{4}-\tau \rho ^{2}+H,
\]%
The TVF equilibrium corresponds to $\rho^2=\tau/2$, $K=0$, $H=\tau^2/4$. 
We now study the function $f(X)$ defined for $X>0$ by
\[
f(X)=-\frac{K^{2}}{X}+X^{2}-\tau X+H,
\]%
and determine the region where $f \ge 0$. Since $f(X)\to -\infty$ as $X\to 0^+$ and $f(X)\to +\infty$ as $X\to +\infty$, the only interesting cases are when $f$
is not monotonous, taking the value $0$ for two different values of $X.$ We
show that this occurs for $(H,K)$ in a certain bounded region of the plane.
The limit cases are when the minimum or maximum of $f$ are on the $X$ axis, so that we
satisfy the system%
\[
f(X)=f^{\prime }(X)=0.
\]%
This gives%
\begin{eqnarray*}
H &=&-3X^{2}+2\tau X \\
K^{2} &=&-2X^{3}+\tau X^{2},
\end{eqnarray*}%
which is the equation, in parametric form, of the curve in the $(H,K)$ plane
where $f(X)$ cancels in a
maximum or a minimum. We observe that we need to study
this curve from $X=0$ until $X=\tau /2$ (TVF flow) for which $K$ cancels (since $K^2$ should be positive), and then take
the symmetric curve through the $H$ axis. The point obtained for $X=\tau /3$
is singular, giving maximum values for $H$ and $K:(H,K)_{\max }=(\tau
^{2}/3,(\tau /3)^{3/2})$. The curve reaches the $H$ axis ($K=0$) for $H=0,$ and $%
H=\tau ^{2}/4$ (this corresponds to the origin and to the heteroclinic in the
plane $(A,A^{\prime })$ found for $K=0).$ We indicate on Figure \ref{caseKnon0} the
shapes of $f(X)$ all along the curve and the shapes of $f$ inside the region
delimited by the curve. 

For $\rho (y)$, the discussion is the same as for the
pendulum equation. Inside the region there is a family of periodic
solutions, depending on $(H,K),$ with amplitudes varying from $0$ for
the left boundary with lower values of $H,$ to a limit corresponding to a family
of homoclinic solutions at the right boundary with larger values of $H=H_{max}(K).$ 
\begin{figure}[th]
\begin{center}
\includegraphics[width=5cm]{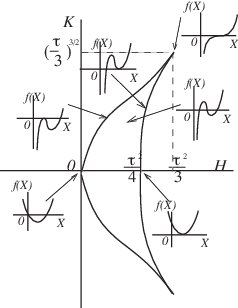}
\end{center}
\caption{$f(X)$ for $K\ne  0$.}
\label{caseKnon0}
\end{figure}
When $\rho $ (which stays $>0$) describes one of the homoclinics, this
corresponds to solutions with the argument $\theta (y)$ such that 
\[
\theta ^{\prime }(y)\underset{y\rightarrow \pm \infty }{\rightarrow }\frac{K%
}{\rho _{\infty }^{2}}=\beta ,
\]%
where the limit is exponential, hence%
\[
\theta (y)=\beta y+\theta _{0}+K\int_{0}^{y}\left( \frac{1}{\rho ^{2}(s)}-%
\frac{1}{\rho _{\infty }^{2}}\right) ds.
\]%
The condition for being an eligible solution (i.e. periodic in the azimuthal
coordinate) is that, there exists $n\in 
\mathbb{N}
$ such that%
\begin{equation}
K\int_{-\infty }^{\infty }\left( \frac{1}{\rho ^{2}(s)}-\frac{1}{\rho
_{\infty }^{2}}\right) ds=2n\pi .  \label{eligiblecond}
\end{equation}%
Since $\rho (y)\leq \rho _{\infty }$ depends on $(H_{\max }(K),K),$ the
condition (\ref{eligiblecond}) might be realized for a discrete set of
values of $K.$ For $K$ close to $0$, $\rho (y)$ stays nearly constant close
to TVF solution, except for a relatively small interval of values for $y$
(see Figure \ref{rhotheta}), so that the corresponding flow looks like a wavy vortex flow
with a small azimuthal modulation which has a little step near $y=0.$

For the periodic solutions $\rho (y)$ of period $T(H,K)$, we obtain%
\begin{eqnarray*}
\theta (y) &=&\beta ^{\prime }y+\theta _{0}+\phi (y), \\
\phi ^{\prime }(y) &=&\frac{K}{\rho ^{2}(s)}-\beta ^{\prime }, \\
\beta ^{\prime } &=&\frac{K}{T(H,K)}\int_{0}^{T(H,K)}\frac{dy}{\rho ^{2}(y)},
\end{eqnarray*}%
where $\phi $ has the same period $T(H,K)$ as $\rho .$ For solutions $\rho (y)$
close to some homoclinic, the amplitude stays nearly constant for long
intervals of $y$ and $\theta $ is the sum of $\beta ^{\prime }y$ plus
periodic bumps as indicated on Figure \ref{rhotheta}. 

 It should be noted that all steady solutions, obtained for $H=H_{min}$ or $H=H_{max}$ for (\ref{steadyamplitude}), correspond to bifurcating time periodic solutions of the Navier-Stokes equations, under the form of classical wavy vortices. Other solutions are new, more exotic and still time periodic in the frame of an observer.  

The persistence of the above solutions under perturbation terms of higher order comes from the study (not made here) with the help of spatial dynamics (where $y$ plays the role of time). Here we obtain a fourth-order ODE of the form (\ref{steadyamplitude}) with higher order terms respecting the symmetries $A\rightarrow Ae^{ih}$, $A\rightarrow \overline{A}$, and reversibility symmetry $(y,A(y))\rightarrow(-y,A(-y))$. This is analogous (however simpler) to what is done in \cite{Iooss2}, see also \cite{Haragus}  p.222. Here the linear situation corresponds to a quadruple $0$ eigenvalue with two $2 \times 2$ Jordan blocks where the principal part of the normal form is given by (\ref{steadyamplitude}). The periodic solutions we find, persist in using an implicit function argument in a suitable function space. The homoclinic solution also persists by the same type of geometric argument as developed in \cite{Haragus}  p.222. More delicate is the problem of persistence under perturbation of the solutions corresponding to the region of the parameter plane $(H,K)$ at the interior of the bounded area. We may observe that these solutions are such that, for a suitable $\beta$, the function $D(y)=A(y)e^{-i\beta y}$ is periodic, with the period of $\rho(y)$. Then the persistence proof applies for these periodic solutions (of a slightly modified equation), with a suitable $\beta$.
It results that the solutions we provide here are principal parts of solutions of the limit Navier-Stokes equations (\ref{newnonlinear}). 

These results are summed up in Theorem \ref{maintheo}.

\begin{figure}[th]
\includegraphics[width=5cm]{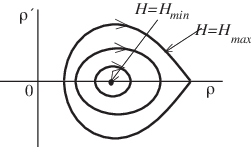}
\includegraphics[width=6cm]{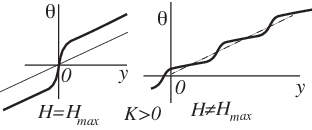}
\caption{Bounded solutions $\rho$  for $K>0$ and $\theta(y)$ for $H_{min}<H\leq H_{max}$.}
\label{rhotheta}
\end{figure}

%%%%%%%%%%%%%%%%%%%%%%%%%%%%%%%%%%%%%%%%%%%%%%
\section{Appendix A: Some computations}
%%%%%%%%%%%%%%%%%%%%%%%%%%%%%%%%%%%%%%%%%%%%%%
\subsection{Rescaled Navier-Stokes equations}\label{App1}

In this section, we provide detail computations for the derivation of \eqref{initial} in Section \ref{sec2}
and in particular we expand $\Delta_{cyl} u$, $r^{-2} u$, $U' + \Omega_{rot}$, $\Omega_{rot}$ in $\eta$ and $\mu$.

\begin{itemize}

\item Expansion of $\Delta_{cyl} u$: we have
$$
\Delta_{cyl} u = \Big(\frac{\nu}{d^3} \frac{\partial^2}{\partial \hat{x}^2}  
+ \frac{1}{\hat{x} d + \bar{r}} \frac{\nu}{d^2} \frac{\partial}{\partial \hat{x}} 
+ \frac{\nu}{d^3} \frac{\bar{r}^2}{(\hat{x} d + \bar{r})^2} \frac{\partial^2}{\partial \hat{y}^2} 
+ \frac{\nu}{d^3}\frac{\partial^2}{\partial \hat{z}^2} \Big) \hat{u}.
$$
We have
$$
\frac{d}{\hat{x} d + \bar{r}} = \frac{1-\eta}{\hat{x} (1 - \eta) + \frac{1+\eta}{2}} = O(1 - \eta),
$$
\begin{align*}
    \frac{\bar{r}^2}{(\hat{x} d + \bar{r})^2} 
    =& \frac{(\frac{1+\eta}{2})^2}{(\hat{x} (1 - \eta) + \frac{1+\eta}{2})^2}
    = 1 - \frac{\hat{x}(1-\eta) (\hat{x}(1-\eta) + (1+\eta))}{(\hat{x} (1 - \eta) + \frac{1+\eta}{2})^2}\\
    =& 1 + O(1 - \eta),
\end{align*}
therefore,
$$
\Delta_{cyl} u = \frac{\nu}{d^3} \Big( \frac{\partial^2}{\partial \hat{x}^2}  + \frac{\partial^2}{\partial \hat{y}^2} 
+\frac{\partial^2}{\partial \hat{z}^2} \Big) \hat{u} 
+ O(1 - \eta) \frac{\nu}{d^3} \Big( \frac{\partial \hat{u}}{\partial \hat{x}} + \frac{\partial^2 \hat{u}}{\partial \hat{y}^2}\Big).
$$

\item Expansion of $r^{-2} u$:

\begin{align*}
    \frac{u}{r^2} = \frac{\nu}{d^3} \frac{d^2}{(\hat{x} d + \bar{r})^2} \hat{u} = O(1-\eta)^2 \frac{\nu}{d^3} \hat{u}.
\end{align*}

\item Study of $U' + \Omega_{rot}$: 

{\begin{align*}
    U' + \Omega_{rot} =& 2 (A - \Omega_{rf})
    = 2 \omega_i \Big( \frac{\mu - \eta^2}{1- \eta^2} - \frac{1 + \mu}{2} \Big)
    = - \omega_i \frac{(1 - \mu)(1 + \eta^2)}{1 - \eta^2}.
\end{align*}}

\item Expansion of $\Omega_{rot}$:

\begin{align*}
     \Omega_{rot} =&  (A - \Omega_{rf}) + \frac{B}{r^2}
    = - \omega_i \frac{(1 - \mu)(1 + \eta^2)}{1 - \eta^2} + \omega_i \frac{1-\mu}{1-\eta^2} \frac{r_i^2}{r^2}\\
    =& - \omega_i \frac{1-\mu}{1-\eta^2} \Big( \frac{1 + \eta^2}{2} - \frac{r_i^2}{r^2} \Big).
\end{align*}
Note that
\begin{align*}
     \frac{r_i^2}{r^2} =& \frac{r_i^2}{(\hat{x} d + \bar{r})^2}
     =  \frac{\eta^2}{(\hat{x} (1 - \eta) + \frac{1+\eta}{2})^2}
     = \Big( \frac{\eta}{(\hat{x} + \frac{1}{2}) (1 - \eta) + \eta} \Big)^2\\
     =& \Big( 1 + \frac{1-\eta}{\eta} (\hat{x} + \frac{1}{2})  \Big)^{-2}\\
     =& 1 - 2  \Big(\frac{1-\eta}{\eta} (\hat{x} + \frac{1}{2})  \Big) + 3  \Big( \frac{1-\eta}{\eta} (\hat{x} + \frac{1}{2})  \Big)^2 + \cdots.
\end{align*}
Hence,
\begin{align*}
    \frac{1}{1-\eta^2} \Big( \frac{1 + \eta^2}{2} - \frac{r_i^2}{r^2} \Big) 
    =& \frac{1}{1-\eta^2} \Big[ \frac{\eta^2 - 1}{2} + 2  \Big(\frac{1-\eta}{\eta} (\hat{x} 
    + \frac{1}{2})  \Big) - 3  \Big( \frac{1-\eta}{\eta} (\hat{x} + \frac{1}{2})  \Big)^2 + \cdots \Big]\\
    =& - \frac{1}{2} + \frac{1}{\eta(1+\eta)}(2 \hat{x} + 1) + O(1 - \eta).
\end{align*}
As
$$\frac{2}{\eta(1+\eta)} - 1 = - \frac{(\eta+2)(\eta - 1)}{\eta (1+\eta)} = O(1-\eta),$$
we have
$$
\frac{1}{1-\eta^2} \Big( \frac{1 + \eta^2}{2} - \frac{r_i^2}{r^2} \Big) = \hat{x} + O(1 - \eta)
$$
and hence
$$
\Omega_{rot} = - \omega_i (1 - \mu)  \Bigl[ \hat{x} + O(1 - \eta)\Bigr].
$$
\end{itemize}

\subsection{Critical Taylor number in axi-symmetric case}\label{App2}
    
Solutions of (\ref{order6}) are linear combinations of $\exp(\lambda_j x)$, where
the $\Lambda=\lambda_j$ are the $6$ complex solutions of
\begin{equation} \label{firstinst4}
(\lambda + \alpha^2-\Lambda^2 )^2 ( \alpha^2 - \Lambda^2) = T \alpha^2
\end{equation}
provided this polynomial of degree $6$ has $6$ different roots. If some roots are equal,
then the solution involves functions of the form $x \exp(\lambda_j x)$. Note that by previous Lemma, $\sigma>\alpha^2$.

This equation is a polynomial of degree $3$ in $\Lambda^2$. It can thus be explicitly solved
and has three roots $\Lambda_k(\lambda,\alpha,T)$ for $1 \le k \le 3$, leading to the six roots
$\pm \Lambda_k^{1/2}(\lambda,\alpha,T)$ for (\ref{firstinst4}).
If these $6$ roots are different,  $u_x$ is of the form
$$
u_x(x) = \sum_{j=1}^6 \beta_j e^{\lambda_j x},
$$
The boundary conditions give
$$
\sum_{j=1}^6 \beta_j e^{\pm \lambda_j/2} = 0,
$$
$$
\sum_{j=1}^6 \beta_j \lambda_j e^{\pm \lambda_j/2} = 0,
$$
$$
\sum_{j=1}^6 \beta_j (\lambda+ \alpha^2 - \lambda_j^2) (\alpha^2 - \lambda_j^2) e^{\pm \lambda_j/2} = 0
$$
which is a $6 \times 6$ linear system.
Let $\mathfrak{D}(\lambda)$ be its determinant.
Then if $\mathfrak{D}(\lambda)$ vanishes and if all the roots $\lambda_j$ are different, $\lambda$ is an 
eigenvalue.

When $\lambda = 0$, equation (\ref{firstinst4}) gives 
$$
\Lambda ^{2} =\alpha ^{2}-\sigma (1)^{1/3}, \quad \sigma = (T\alpha ^{2})^{1/3}
$$
where $1^{1/3}$ stands for the three different cubic roots of $1$.
Thus we have 
\begin{eqnarray*}
\lambda _{1}^{2} &=&\alpha ^{2}-\sigma , \\
\lambda _{2}^{2} &=&\alpha ^{2}-\sigma j, \\
\lambda _{3}^{2} &=&\alpha ^{2}-\sigma j^{2},
\end{eqnarray*}%
where%
\[
j=-\frac{1}{2}+i\frac{\sqrt{3}}{2}.
\]%
One can see that $\lambda_1$, $\lambda_2$, and $\lambda_3$ are always different. Therefore,
an eigenvector $u_x$ should take the form 
\[
u_x =A_{1}e^{\lambda _{1}x}+B_{1}e^{-\lambda _{1}x}+A_{2}e^{\lambda
_{2}x}+B_{2}e^{-\lambda _{2}x}+A_{3}e^{\lambda _{3}x}+B_{3}e^{-\lambda _{3}x}
\]%
with%
\begin{align*}
\lambda _{1} &=ib_{1}, \\
\lambda _{2} &=a_{2}+ib_{2}, \\
\lambda _{3} &=a_{2}-ib_{2},
\end{align*}%
where $b_1$, $a_2$ and $b_2$ are given by (\ref{lambda_j}).

The boundary conditions are
$$
u_x (\pm 1/2) = Du_x (\pm 1/2) = (\alpha ^{2}-D^{2})^{2}u_x (\pm 1/2) =0.
$$
We may observe that%
\begin{eqnarray*}
(\alpha ^{2}-\lambda _{1}^{2})^{2} &=&\sigma ^{2}, \\
(\alpha ^{2}-\lambda _{2}^{2})^{2} &=&\sigma ^{2}j^{2}, \\
(\alpha ^{2}-\lambda _{3}^{2})^{2} &=&\sigma ^{2}j,
\end{eqnarray*}%
hence the boundary conditions become%
\begin{eqnarray*}
A_{1}e^{\lambda _{1}/2}+B_{1}e^{-\lambda _{1}/2}+A_{2}e^{\lambda
_{2}/2}+B_{2}e^{-\lambda _{2}/2}+A_{3}e^{\lambda _{3}/2}+B_{3}e^{-\lambda
_{3}/2} &=&0, \\
A_{1}e^{-\lambda _{1}/2}+B_{1}e^{\lambda _{1}/2}+A_{2}e^{-\lambda
_{2}/2}+B_{2}e^{\lambda _{2}/2}+A_{3}e^{-\lambda _{3}/2}+B_{3}e^{\lambda
_{3}/2} &=&0, \\
\lambda _{1}A_{1}e^{\lambda _{1}/2}-\lambda _{1}B_{1}e^{-\lambda
_{1}/2}+\lambda _{2}A_{2}e^{\lambda _{2}/2}-\lambda _{2}B_{2}e^{-\lambda
_{2}/2}+\lambda _{3}A_{3}e^{\lambda _{3}/2}-\lambda _{3}B_{3}e^{-\lambda
_{3}/2} &=&0, \\
\lambda _{1}A_{1}e^{-\lambda _{1}/2}-\lambda _{1}B_{1}e^{\lambda
_{1}/2}+\lambda _{2}A_{2}e^{-\lambda _{2}/2}-\lambda _{2}B_{2}e^{\lambda
_{2}/2}+\lambda _{3}A_{3}e^{-\lambda _{3}/2}-\lambda _{3}B_{3}e^{\lambda
_{3}/2} &=&0, \\
A_{1}e^{\lambda _{1}/2}+B_{1}e^{-\lambda _{1}/2}+j^{2}A_{2}e^{\lambda
_{2}/2}+j^{2}B_{2}e^{-\lambda _{2}/2}+jA_{3}e^{\lambda
_{3}/2}+jB_{3}e^{-\lambda _{3}/2} &=&0, \\
A_{1}e^{-\lambda _{1}/2}+B_{1}e^{\lambda _{1}/2}+j^{2}A_{2}e^{-\lambda
_{2}/2}+j^{2}B_{2}e^{\lambda _{2}/2}+jA_{3}e^{-\lambda
_{3}/2}+jB_{3}e^{\lambda _{3}/2} &=&0.
\end{eqnarray*}%
For obtaining a non trivial solution we need to verify that
\[
\Delta _{\alpha ,\tau }=\left\vert 
\begin{array}{cccccc}
e^{\lambda _{1}/2} & e^{\lambda _{2}/2} & e^{\lambda _{3}/2} & e^{-\lambda
_{1}/2} & e^{-\lambda _{2}/2} & e^{-\lambda _{3}/2} \\ 
\lambda _{1}e^{\lambda _{1}/2} & \lambda _{2}e^{\lambda _{2}/2} & \lambda
_{3}e^{\lambda _{3}/2} & -\lambda _{1}e^{-\lambda _{1}/2} & -\lambda
_{2}e^{-\lambda _{2}/2} & -\lambda _{3}e^{-\lambda _{3}/2} \\ 
e^{\lambda _{1}/2} & j^{2}e^{\lambda _{2}/2} & je^{\lambda _{3}/2} & 
e^{-\lambda _{1}/2} & j^{2}e^{-\lambda _{2}/2} & je^{-\lambda _{3}/2} \\ 
e^{-\lambda _{1}/2} & e^{-\lambda _{2}/2} & e^{-\lambda _{3}/2} & e^{\lambda
_{1}/2} & e^{\lambda _{2}/2} & e^{\lambda _{3}/2} \\ 
\lambda _{1}e^{-\lambda _{1}/2} & \lambda _{2}e^{-\lambda _{2}/2} & \lambda
_{3}e^{-\lambda _{3}/2} & -\lambda _{1}e^{\lambda _{1}/2} & -\lambda
_{2}e^{\lambda _{2}/2} & -\lambda _{3}e^{\lambda _{3}/2} \\ 
e^{-\lambda _{1}/2} & j^{2}e^{-\lambda _{2}/2} & je^{-\lambda _{3}/2} & 
e^{\lambda _{1}/2} & j^{2}e^{\lambda _{2}/2} & je^{\lambda _{3}/2}%
\end{array}%
\right\vert 
\]
vanishes. We observe that $\Delta _{\alpha ,\tau }$ is real since $\lambda _{3}=%
\overline{\lambda _{2}}$. 
Now, after reorganisation, we have
\[
\Delta _{\alpha ,\tau }=\left\vert 
\begin{array}{cccccc}
e^{\lambda _{1}/2} & e^{\lambda _{2}/2} & e^{\lambda _{3}/2} & e^{-\lambda
_{1}/2} & e^{-\lambda _{2}/2} & e^{-\lambda _{3}/2} \\ 
0 & (1-\frac{\lambda _{2}}{\lambda _{1}})e^{\lambda _{2}/2} & (1-\frac{%
\lambda _{3}}{\lambda _{1}})e^{\lambda _{3}/2} & 2e^{-\lambda _{1}/2} & (1+%
\frac{\lambda _{2}}{\lambda _{1}})e^{-\lambda _{2}/2} & (1+\frac{\lambda _{3}%
}{\lambda _{1}})e^{-\lambda _{3}/2} \\ 
0 & (1-j^{2})e^{\lambda _{2}/2} & (1-j)e^{\lambda _{3}/2} & 0 & 
(1-j^{2})e^{-\lambda _{2}/2} & (1-j)e^{-\lambda _{3}/2} \\ 
e^{-\lambda _{1}/2} & e^{-\lambda _{2}/2} & e^{-\lambda _{3}/2} & e^{\lambda
_{1}/2} & e^{\lambda _{2}/2} & e^{\lambda _{3}/2} \\ 
2e^{-\lambda _{1}/2} & (1+\frac{\lambda _{2}}{\lambda _{1}})e^{-\lambda
_{2}/2} & (1+\frac{\lambda _{3}}{\lambda _{1}})e^{-\lambda _{3}/2} & 0 & (1-%
\frac{\lambda _{2}}{\lambda _{1}})e^{\lambda _{2}/2} & (1-\frac{\lambda _{3}%
}{\lambda _{1}})e^{\lambda _{3}/2} \\ 
0 & (1-j^{2})e^{-\lambda _{2}/2} & (1-j)e^{-\lambda _{3}/2} & 0 & 
(1-j^{2})e^{\lambda _{2}/2} & (1-j)e^{\lambda _{3}/2}%
\end{array}%
\right\vert ,
\]%
which takes the form%
\[
\Delta _{\alpha ,\tau }=\left\vert 
\begin{array}{cc}
M & N \\ 
N & M%
\end{array}%
\right\vert 
\]%
with%
\begin{eqnarray*}
M &=&\left( 
\begin{array}{ccc}
e^{\lambda _{1}/2} & e^{\lambda _{2}/2} & e^{\lambda _{3}/2} \\ 
0 & (1-\frac{\lambda _{2}}{\lambda _{1}})e^{\lambda _{2}/2} & (1-\frac{%
\lambda _{3}}{\lambda _{1}})e^{\lambda _{3}/2} \\ 
0 & (1-j^{2})e^{\lambda _{2}/2} & (1-j)e^{\lambda _{3}/2}%
\end{array}%
\right) , \\
N &=&\left( 
\begin{array}{ccc}
e^{-\lambda _{1}/2} & e^{-\lambda _{2}/2} & e^{-\lambda _{3}/2} \\ 
2e^{-\lambda _{1}/2} & (1+\frac{\lambda _{2}}{\lambda _{1}})e^{-\lambda
_{2}/2} & (1+\frac{\lambda _{3}}{\lambda _{1}})e^{-\lambda _{3}/2} \\ 
0 & (1-j^{2})e^{-\lambda _{2}/2} & (1-j)e^{-\lambda _{3}/2}%
\end{array}%
\right) .
\end{eqnarray*}
We have
\begin{align*}
    \Delta _{\alpha ,\tau } &= 
    \begin{vmatrix}
        M+N & M+N\\
        N& M
    \end{vmatrix}
    = \begin{vmatrix}
        M+N & 0\\
        N& M-N
    \end{vmatrix}
= |M+N| |M-N|.
\end{align*}
Thus $\Delta_{\alpha,\tau} = 0$ if and only if $|M+N| = 0$ or $|M-N|=0$.

With the ``+" sign we obtain%
\[
\left\vert \left( 
\begin{array}{ccc}
\cos \frac{b_{1}}{2} & \cosh \frac{\lambda _{2}}{2} & \cosh \frac{\lambda
_{3}}{2} \\ 
e^{-ib_{1}/2} & \cosh \frac{\lambda _{2}}{2}-\frac{\lambda _{2}}{\lambda _{1}%
}\sinh \frac{\lambda _{2}}{2} & \cosh \frac{\lambda _{3}}{2}-\frac{\lambda
_{3}}{\lambda _{1}}\sinh \frac{\lambda _{3}}{2} \\ 
0 & (1-j^{2})\cosh \frac{\lambda _{2}}{2} & (1-j)\cosh \frac{\lambda _{3}}{2}%
\end{array}%
\right) \right\vert =0
\]%
i.e.%
\begin{equation}
\left\vert \left( 
\begin{array}{ccc}
\cos \frac{b_{1}}{2} & \cosh \frac{\lambda _{2}}{2} & \cosh \frac{\lambda
_{3}}{2} \\ 
-b_{1}\sin \frac{b_{1}}{2} & \lambda _{2}\sinh \frac{\lambda _{2}}{2} & 
\lambda _{3}\sinh \frac{\lambda _{3}}{2} \\ 
0 & -j\cosh \frac{\lambda _{2}}{2} & j^{2}\cosh \frac{\lambda _{3}}{2}%
\end{array}%
\right) \right\vert =0,  \label{equ +}
\end{equation}%
where 
\begin{eqnarray*}
\lambda _{2} &=&a_{2}+ib_{2}, \\
\lambda _{3} &=&a_{2}-ib_{2},
\end{eqnarray*}%
and where $a_{2},b_{1},b_{2}$ are functions of $\alpha ^{2}$ and $\sigma $ defined by (%
\ref{lambda_j}). A direct computation of this determinant leads to $\mathfrak{F}_1 = 0$.

With the  ``-" sign, we obtain%
\begin{equation}
\left\vert  
\begin{array}{ccc}
\sin \frac{b_{1}}{2} & \sinh \frac{\lambda _{2}}{2} & \sinh \frac{\lambda
_{3}}{2} \\ 
b_{1}\cos \frac{b_{1}}{2} & \lambda _{2}\cosh \frac{\lambda _{2}}{2} & 
\lambda _{3}\cosh \frac{\lambda _{3}}{2} \\ 
0 & -j\sinh \frac{\lambda _{2}}{2} & j^{2}\sinh \frac{\lambda _{3}}{2}%
\end{array}%
 \right\vert =0,  \label{equ -}
\end{equation}%
which leads to $\mathfrak{F}_2 = 0$.

The equation $\mathfrak{F}_1 = 0$ comes from
\[
|M+N|=0
\]
which corresponds to solutions such that 
\[
A_1=B_1,A_2=B_2,A_3=B_3
\]
in the eigenvector $u_x$. This means that in this case $u_x$ and $u_y$ are even in $x$. In the same way we prove that $\mathfrak{F}_2 = 0$ gives $u_x$ and $u_y$ odd in $x$. It is then clear that the solutions for $\mathfrak{F}_1 = 0$ and $\mathfrak{F}_2 = 0$ are distinct.

\subsection{Coefficients of the Landau equation}\label{coefLandauequ}

Let us consider the system (\ref{newnonlinear}) ruling the velocity field $U=(u_{\bot},\widehat{u}_y)$ independent of $y$,
then suppressing the hat, we obtain (below we detail the projection $\Pi$ on divergence free vector fields): 
\begin{eqnarray*}
\partial_t U &=&\Delta _{\bot }U-\nabla _{\bot }p+\left( 
\begin{array}{c}
Tu_{y} \\ 
u_{x} \\ 
0%
\end{array}%
\right) -(U_{\bot }\cdot \nabla _{\bot })U, \\
\nabla _{\bot }\cdot U_{\bot } &=&0,
\end{eqnarray*}%
where the subscipt $\bot $ means components $(x,z).$ The field $U$ is $2\pi
/\alpha _{c}$ periodic in $z,$ and satisfies the $0$ boundary conditions at $%
x=\pm 1/2$. Let us define operators%
\[
\mathbf{L}_{0}U=\Delta _{\bot }U-\nabla _{\bot }p_{0}+\left( 
\begin{array}{c}
T_{c}u_{y} \\ 
u_{x} \\ 
0%
\end{array}%
\right) ,
\]%
where $p_{0}$ is such that 
\[
\nabla _{\bot }\cdot (\mathbf{L}_{0}U)_{\bot }=0,\]
\[
(\mathbf{L}_{0}U)_{x}=0, \quad x=\pm1/2;
\]
and
\[
\mathbf{L}_{1}U=-\nabla _{\bot }p_{1}+\left( 
\begin{array}{c}
u_{y} \\ 
0 \\ 
0%
\end{array}%
\right) ,
\]%
with $p_{1}$ such that%
\[
\nabla _{\bot }\cdot (\mathbf{L}_{1}U)_{\bot }=0,
\]
\[
(\mathbf{L}_{1}U)_{x}=0,\quad x=\pm1/2;
\]
and
\begin{eqnarray*}
\mathbf{B}(U,U) &=&-(U_{\bot }\cdot \nabla _{\bot })U+\nabla _{\bot }q, \\
\nabla _{\bot }\cdot (\mathbf{B}(U,U))_{\bot } &=&0,\text{   }(\mathbf{B}(U,U))_{x}=0,\text{   }x=\pm1/2.
\end{eqnarray*}
Notice that the quadratic operator $\mathbf{B}$ is symmetrized if applied to two different vector functions:
\[
2\mathbf{B}(U,V))=:1/2\{\mathbf{B}(U+V,U+V)-\mathbf{B}(U-V,U-V)\}.
\]
Now our system reads as%
\begin{equation}
\partial_t U=\mathbf{L}_{0}U+\tau \mathbf{L}_{1}U+\mathbf{B}%
(U,U),  \label{NSequ}
\end{equation}%
in a space of divergence free vector fields, satisfying the boundary
conditions, and where 
\[
\tau =T-T_{c}.
\]%
For obtaining the Landau equation describing the first bifurcation occuring
for $\tau $ close to $0$, we use the fact that $\mathbf{L}_{0}$ has a double $%
0$ eigenvalue with eigenvectors%
\[
\zeta =e^{i\alpha _{c}z}u^0(x),\overline{\zeta }=e^{-i\alpha _{c}z}%
\overline{u^0}(x),
\]%
(see subsection \ref{explambda}) and the fact that the rest of the spectrum of the self adjoint operator $%
\mathbf{L}_{0}+\tau \mathbf{L}_{1}$ is only composed of \thinspace $<0$
isolated eigenvalues, not close to 0. Moreover the center manifold reduction
applies (see \cite{Haragus}) with a $O(2)$ symmetry, and the
dynamics near $0$ reduces to the 2-dimensional differential equation in $%
\mathbb{C}$, at main orders
\begin{equation}
\frac{dA}{dt}=a\tau A-cA|A|^{2},  \label{Landauequ}
\end{equation}%
where 
\begin{eqnarray}
U &=&A(t)\zeta +\overline{A}(t)\overline{\zeta }+\Phi (A,\overline{A},\tau ),
\label{centermanif} \\
\Phi (A,\overline{A},\tau ) &=&\tau A\Phi _{10}^{(1)}+A^{2}\Phi
_{20}+|A|^{2}\Phi _{11}+\overline{A}^{2}\overline{\Phi _{20}}+...  \nonumber
\end{eqnarray}%
coefficients $\Phi _{ij}$ being vector functions of $x,z$ satisfying the
boundary conditions. The coefficients $a$, $c$ and $\Phi _{ij}$ may be
computed as indicated below. Replacing $U$ by (\ref{centermanif}) in (\ref%
{NSequ}) and using (\ref{Landauequ}), gives after identification of each
monomial $\tau ^{p}A^{n}\overline{A}^{m}$%
\begin{eqnarray}
a\zeta  &=&\mathbf{L}_{0}\Phi _{10}^{(1)}+\mathbf{L}_{1}\zeta 
\label{coef a} \\
0 &=&\mathbf{L}_{0}\Phi _{20}+\mathbf{B}(\zeta ,\zeta ),  \nonumber \\
0 &=&\mathbf{L}_{0}\Phi _{11}+2\mathbf{B}(\zeta ,\overline{\zeta }) , 
\nonumber \\
-c\zeta  &=&\mathbf{L}_{0}\Phi _{21}+ 2\mathbf{B}(\zeta ,\Phi _{11})+2\mathbf{B}(\overline{\zeta },\Phi _{20}).  \label{coef c}
\end{eqnarray}%
Since we have%
\[
\mathbf{L}_{0}\zeta =0,
\]%
with $\mathbf{L}_{0}$ selfadjoint, we obtain%
\[
a\langle \zeta ,\zeta \rangle =\langle \mathbf{L}_{1}\zeta ,\zeta \rangle ,
\]%
\begin{equation}
-c\langle \zeta ,\zeta \rangle =\langle 2\mathbf{B}(\zeta ,\Phi _{11})+2\mathbf{B}(\overline{\zeta },\Phi _{20}),\zeta \rangle ,  \label{explicit c}
\end{equation}%
where the scalar product is defined by%
\[
\langle U,V\rangle =\int_{0}^{2\pi /\alpha_c }\int_{-1/2}^{1/2}(u_{x}\overline{%
v_{x}}+T_{c}u_{y}\overline{v_{y}}+u_{z}\overline{v_{z}}) \, dxdz.
\]%
Now we show that $a$ is just the coefficient $a_{3}$ computed before, and we
prove below that $c>0,$ without an explicit computation. Indeed we have%
\begin{eqnarray*}
\zeta  &=&e^{i\alpha_c z}\left( 
\begin{array}{c}
u_{x}^{0}(x) \\ 
u_{y}^{0}(x) \\ 
u_{z}^{0}(x)%
\end{array}%
\right) , \\
Du_{x}^{0}+i\alpha_c u_{z}^{0} &=&0,
\end{eqnarray*}%
\begin{eqnarray*}
(\alpha_c ^{2}-D^{2})^{2}u_{x}^{0} &=&T_{c}\alpha_c ^{2}u_{y}^{0}, \\
(\alpha_c ^{2}-D^{2})u_{y}^{0} &=&u_{x}^{0}, \\
u_{x}^{0} &=&Du_{x}^{0}=u_{y}^{0}=0,x=\pm 1/2,
\end{eqnarray*}%
hence%
\[
a\int_{-1/2}^{1/2}
\Bigl( |u_{x}^{0}|^{2}+T_{c}|u_{y}^{0}|^{2}+|u_{z}^{0}|^{2} \Bigr) \, dx
=\int_{-1/2}^{1/2}u_{y}^{0}u_{x}^{0}dx,
\]%
then noticing that $(u_{x}^{0},u_{y}^{0})$ is real and $u_{z}^{0}=(i/\alpha_c
)Du_{x}^{0}$, this is exactly the coefficient $a_{3}>0$ already computed at Lemma \ref{lem4.2}.

Now, for the computation of $c$ we notice that for any \emph{real} vector
field $U$ we have%
\begin{equation}
\langle \mathbf{B}(U,U),U\rangle =0.\label{remarkident}
\end{equation}
Indeed 
\begin{eqnarray*}
-\langle \mathbf{B}(U,U),U\rangle  &=&\langle (U_{\bot }\cdot \nabla _{\bot
})U,U\rangle  \\
&=&\int \Bigl[\nabla _{\bot }\cdot (U_{\bot }\frac{U_{\bot }^{2}}{2}%
)+T_{c}\nabla _{\bot }\cdot (U_{\bot }\frac{U_{y}^{2}}{2}) \Bigr] \, dxdz=0
\end{eqnarray*}%
because of periodicity and boundary conditions. It then results that%
\begin{equation}
2\langle \mathbf{B}(U_{1},U_{2}),U_{1}\rangle +\langle \mathbf{B}(U_{1},U_{1}),U_{2}\rangle =0,
\label{cubic ident}
\end{equation}%
for any real divergence free vector fields $U_{1},U_{2}$ satisfying the
boundary conditions (just take in (\ref{remarkident}) $U=U_{1}+tU_{2}$ and find the coefficient of $%
t$). Coming back to (\ref{explicit c}) we may observe that%
\[
-c\langle \zeta +\overline{\zeta },\zeta +\overline{\zeta }\rangle =\langle 2%
\mathbf{B}(\zeta +\overline{\zeta },\Phi _{2}),\zeta +\overline{\zeta }%
\rangle ,
\]%
where $\Phi _{2}$ is defined by%
\[
\mathbf{L}_{0}\Phi _{2}+\mathbf{B}(\zeta +\overline{\zeta },\zeta +\overline{%
\zeta })=0.
\]%
Then, from (\ref{cubic ident}) we obtain%
\begin{eqnarray*}
c\langle \zeta +\overline{\zeta },\zeta +\overline{\zeta }\rangle 
&=&\langle \mathbf{B}(\zeta +\overline{\zeta },\zeta +\overline{\zeta }%
),\Phi _{2}\rangle  \\
&=&-\langle \mathbf{L}_{0}\Phi _{2},\Phi _{2}\rangle .
\end{eqnarray*}%
Observing that $\Phi _{2}$ is orthogonal to the kernel of $\mathbf{L}_{0},$
where $0$ is the largest eigenvalue of the selfadjoint operator, we claim
that $\langle \mathbf{L}_{0}\Phi _{2},\Phi _{2}\rangle <0.$ Hence $c>0$ (we
may mention that this type of proof is due to V.Yudovich \cite{Yudo3} for the B\'{e}%
nard-Rayleigh convection problem).

For an explicit calculation of $c$ we need to compute $\Phi _{11}$ and $\Phi
_{20}.$ Let us define%
\[
\Phi _{11}=(u_{x}^{11},u_{y}^{11},u_{z}^{11})^{t},
\]%
and%
\[
(\zeta _{\bot }\cdot \nabla _{\bot })\overline{\zeta }+(\overline{\zeta }%
_{\bot }\cdot \nabla _{\bot })\zeta =\left( 
\begin{array}{c}
 2D(u_{x}^{0})^{2} \\ 
2D(u_{x}^{0}u_{y}^{0}) \\ 
0%
\end{array}%
\right) ,
\]%
then we need to solve%
\begin{eqnarray*}
D^{2}u_{x}^{11}-Dp+T_{c}u_{y}^{11} &=& 2D(u_{x}^{0})^{2}, \\
D^{2}u_{y}^{11}+u_{x}^{11} &=&2D(u_{x}^{0}u_{y}^{0}), \\
D^{2}u_{z}^{11} &=&0, \\
Du_{x}^{11} &=&0
\end{eqnarray*}%
with boundary conditions. Finally%
\[
\Phi _{11}=(0,u_{y}^{11},0)^{t},
\]%
with%
\[
D^{2}u_{y}^{11}=2D(u_{x}^{0}u_{y}^{0}),u_{y}^{11}(\pm 1/2)=0,
\]%
hence $u_{y}^{11}$ is odd in $x,$ and%
\begin{equation}
u_{y}^{11}(x)=2\int_{0}^{x}u_{x}^{0}(s)u_{y}^{0}(s)ds-4x%
\int_{0}^{1/2}u_{x}^{0}(s)u_{y}^{0}(s)ds.  \label{Phi11}
\end{equation}%
Let us now define%
\[
\Phi _{20}=e^{2i\alpha_c z}(u_{x}^{20},u_{y}^{20},u_{z}^{20})^{t}
\]%
and compute%
\[
(\zeta _{\bot }\cdot \nabla _{\bot })\zeta =e^{2i\alpha_c z}\left( 
\begin{array}{c}
0 \\ 
u_{x}^{0}Du_{y}^{0}-u_{y}^{0}Du_{x}^{0} \\ 
\frac{i}{\alpha_c }u_{x}^{0}D^{2}u_{x}^{0}-\frac{i}{\alpha_c }(Du_{x}^{0})^{2}%
\end{array}%
\right) ,
\]%
then we need to solve%
\begin{eqnarray*}
(D^{2}-4\alpha_c ^{2})u_{x}^{20}-Dp+T_{c}u_{y}^{20} &=&0, \\
(D^{2}-4\alpha_c ^{2})u_{y}^{20}+u_{x}^{20}
&=&u_{x}^{0}Du_{y}^{0}-u_{y}^{0}Du_{x}^{0}, \\
(D^{2}-4\alpha_c ^{2})u_{z}^{20}-2i\alpha_c p &=&\frac{i}{\alpha_c }%
[u_{x}^{0}D^{2}u_{x}^{0}-(Du_{x}^{0})^{2}], \\
Du_{x}^{20}+2i\alpha_c u_{z}^{20} &=&0,
\end{eqnarray*}%
\[
u_{x}^{20}=Du_{x}^{20}=u_{y}^{20}=0,x=\pm 1/2.
\]%
This leads to the 6th order system%
\begin{eqnarray*}
(D^{2}-4\alpha_c ^{2})^{2}u_{x}^{20}-4\alpha_c ^{2}T_{c}u_{y}^{20}
&=&2D[u_{x}^{0}D^{2}u_{x}^{0}-(Du_{x}^{0})^{2}], \\
(D^{2}-4\alpha_c ^{2})u_{y}^{20}+u_{x}^{20}
&=&u_{x}^{0}Du_{y}^{0}-u_{y}^{0}Du_{x}^{0}, \\
u_{x}^{20} &=&Du_{x}^{20}=u_{y}^{20}=0,x=\pm 1/2.
\end{eqnarray*}%
This gives $u_{x}^{20}$ and $u_{y}^{20}$ real and odd in $x,$ while $%
u_{z}^{20}$ is pure imaginary and even in $x.$ Now we can compute%
\[
2\mathbf{B}(\zeta ,\Phi _{11})=-e^{i\alpha_c z}\left( 
\begin{array}{c}
0 \\ 
u_{x}^{0}Du_{y}^{11} \\ 
0%
\end{array}%
\right) ,
\]%
which satisfies the divergence free (in $(x,z)$) and boundary conditions. We also obtain
\begin{align*}
&2\mathbf{B}(\overline{\zeta },\Phi _{20})\\
&=-\Pi e^{i\alpha_c z}\left( 
\begin{array}{c}
D(u_{x}^{0}u_{x}^{20})+2u_{x}^{20}Du_{x}^{0}+\frac{1}{2}u_{x}^{0}Du_{x}^{20}
\\ 
u_{x}^{0}Du_{y}^{20}+2u_{y}^{20}Du_{x}^{0}+u_{x}^{20}Du_{y}^{0}+\frac{1}{2}%
u_{y}^{0}Du_{x}^{20} \\ 
\frac{i}{2\alpha_c }[D(u_{x}^{0}Du_{x}^{20})-2u_{x}^{20}D^{2}u_{x}^{0}]%
\end{array}%
\right).
\end{align*}
Hence, from (\ref{explicit c}), we finally obtain the positive number $c$ which is given by
\begin{eqnarray*}
&&c\int_{-1/2}^{1/2} \Bigl(|u_{x}^{0}|^{2}+T_{c}|u_{y}^{0}|^{2}+\frac{1}{\alpha_c
^{2}}|Du_{x}^{0}|^{2} \Bigr)  \, dz \, dx 
\\
&=&%
\int_{-1/2}^{1/2}T_{c}u_{y}^{0}
\Bigl[ u_{x}^{0}Du_{y}^{11}+u_{x}^{0}Du_{y}^{20}+2u_{y}^{20}Du_{x}^{0}+u_{x}^{20}Du_{y}^{0}+%
\frac{1}{2}u_{y}^{0}Du_{x}^{20} \Bigr] \, dx \\
&&+\int_{-1/2}^{1/2}u_{x}^{0} \Bigl[D(u_{x}^{0}u_{x}^{20})+2u_{x}^{20}Du_{x}^{0}+%
\frac{1}{2}u_{x}^{0}Du_{x}^{20} \Bigr] \, dx \\
&&+\int_{-1/2}^{1/2}\frac{1}{2\alpha_c ^{2}}%
Du_{x}^{0} \Bigl[ D(u_{x}^{0}Du_{x}^{20})-2u_{x}^{20}D^{2}u_{x}^{0} \Bigr] \, dx
\end{eqnarray*}
since the projection $\Pi $ disappears in the scalar product.

%%%%%%%%%%%%%%%%%%%%%%%%%%%%

\subsection{Computation of the Ginzburg-Landau equation}\label{compGLequ}

%%%%%%%%%%%%%%%%%%%%%%%%%%%%

In this subsection we derive the partial differential equation satisfied by
the amplitude $A(y,t),$ analogue to the usually called Ginzburg-Landau
equation.

The perturbation $U$ of the Couette flow is decomposed as%
\begin{equation}
U=A\zeta +\overline{A}\overline{\zeta }+\Phi (\tau ,A,\overline{A},\partial
_{y})  \label{decomp U}
\end{equation}%
where%
\[
A=A(y,t),\zeta =e^{i\alpha _{c}z}u^0(x), 
\]%
and where we consider $y$ as a "slow variable" in the amplitude $A.$ Hence $%
\partial _{y}$ is considered as a small parameter, and $A,\overline{A}%
,\partial _{y}A,\partial _{y}\overline{A},\partial _{y}^{2}A,\partial
_{y}^{2}\overline{A}...$ are considered to be independent in the expansions (see \cite{Iooss} and \cite{Iooss2} for the justification of such a computation in the search of steady solutions for Navier-Stokes equations). 
$\Phi $ replaces the usual expression of the center manifold, which was used
before for computing the coefficients in the Landau equation (see (\ref{coef a}),
(\ref{coef c})). The system with which we start is (\ref{newnonlinear}) where the hats are suppressed:%
\begin{equation}
\partial_t U=\mathbf{L}_{0}U+\tau \mathbf{L}_{1}U+\mathbf{B}%
(U,U),  \label{basic NS syst}
\end{equation}%
\[
\mathbf{L}_{0}U=\Delta_{\bot} U+\nabla_{\bot} p_{0}+\mathfrak{R}x\frac{\partial U}{%
\partial y}+\left( 
\begin{array}{c}
T_{c}u_{y} \\ 
u_{x} \\ 
0%
\end{array}%
\right) , 
\]%
\[
\mathbf{L}_{1}U=\left( 
\begin{array}{c}
u_{y} \\ 
0 \\ 
0%
\end{array}%
\right) +\nabla p_{1}, 
\]%
\[
\mathbf{B}(U,U)=-(U_{\bot}\cdot \nabla_{\bot} )U+\nabla p_{2}, 
\]%
where%
\begin{eqnarray*}
\nabla _{\bot}\cdot U_{\bot}+\mathfrak{R}\partial_yu_y =0,&&\text{   }U|_{x=\pm 1/2}=0, \\
\tau &=&2\mathfrak{R}(\omega -\omega _{c})=T-T_{c},
\end{eqnarray*}%
the terms $\nabla p_{0},\nabla p_{1},\nabla p_{2}$ are such that the result
(the left hand side) is divergence free and satisfies the boundary condition
required for the projection $\Pi .$ By construction, we have for the
eigenvector $\zeta :$%
\[
\mathbf{L}_{0}\zeta =0,
\]%
with%
\[
u^0(x)=(u_{x}^{0},u_{y}^{0},u_{z}^{0}), 
\]%
where $u_{x}^{0},u_{y}^{0}$ are real and even in $x,$ while $
u_{z}^{0}$ is pure imaginary and odd in $x.$

The partial differential system ruling $A(y,t)$ and $\overline{A}(y,t)$ has
to commute with the symmetries commuting with the system (\ref{basic NS syst}%
). There are four important such symmetries:

\begin{itemize}

\item[i)] translation $z\mapsto z+h,$ represented by the operator $T_{h},$

\item[ii)] symmetry through horizontal plane: $z\mapsto -z$ represented by the
symmetry $S_{0},$

\item[iii)] azimuthal rotation: $y\mapsto y+b$ represented by the operator $R_{b},$

\item[iv)] symmetry with respect to $z$ axis: $(x,y)\mapsto (-x,-y)$ represented by
the symmetry $S_{1}.$ It should be noted that the symmetry $S_{1}$ is new
with respect to the usual Couette-Taylor system. This is due to our
asymptotic geometric situation and to the fact that we sit in a suitable rotating frame.

\end{itemize}

Then, we have 
\begin{eqnarray*}
T_{h}U(x,y,z) &=&U(x,y,z+h), \\
S_{0}U(x,y,z) &=&(u_{x},u_{y},-u_{z})(x,y,-z), \\
R_{b}U(x,y,z) &=&U(x,y+b,z), \\
S_{1}U(x,y,z) &=&(-u_{x},-u_{y},u_{z})(-x,-y,z),
\end{eqnarray*}%
which corresponds to require the commutation of the amplitude system
respectively with the following actions:%
\begin{eqnarray*}
A &\mapsto &Ae^{i\alpha h}, \\
A &\mapsto &\overline{A}, \\
A(y) &\mapsto &A(y+b), \\
A(y) &\mapsto &-A(-y).
\end{eqnarray*}%
This explains why the principal terms in the PDE system are as%
\begin{equation}
\frac{\partial A}{\partial t}=a\tau A-b\frac{\partial ^{2}A}{\partial y^{2}}%
-cA|A|^{2}+h.o.t.  \label{GL equaton}
\end{equation}%
where coefficients are real. Higher order terms would be in $\tau ^{2}A,$ $%
\frac{\partial ^{4}A}{\partial y^{4}},$ $\tau \frac{\partial ^{2}A}{\partial
y^{2}},$ $A|A|^{4},$ $\tau A|A|^{2},$ $|A|^{2}\frac{\partial ^{2}A}{\partial
y^{2}},...$ Below we explain how to compute coefficients containing $y$
derivatives in (\ref{GL equaton})$)$, here only coefficient $b$ since we
already gave the method for computing the coefficients with no $y$
derivative. It should be noticed that if we consider the linearized system
near the Couette flow, we need to recover the most at right eigenvalues $%
\lambda $ and $\overline{\lambda }$ computed before where we assume a
dependency in $e^{i\beta y}$ for $A.$ We then obtain%
\[
\lambda =a\tau -b\beta ^{2}+h.o.t.
\]%
and $\overline{\lambda }.$ This is exactly what we already obtained at Lemma \ref{lem4.2}, 
with the explicit dependency of coefficient $b$ on the Reynolds number (%
$a=a_{3},b=b_{4}\mathfrak{R}^{2}).$

For the computation let us define the coefficients as%
\[
\Phi =\sum_{m\geq 0,|p|+|q|\geq 1}\Phi _{pq}^{(m)}(x,z)\tau ^{m}A^{p_{0}}%
\overline{A}^{q_{0}}(\partial _{y}A)^{p_{1}}(\partial _{y}\overline{A}%
)^{q_{1}}(\partial _{y}^{2}A)^{p_{2}}(\partial _{y}^{2}\overline{A}%
)^{q_{2}}..
\]%
where%
\begin{eqnarray*}
p &=&(p_{0},p_{1},p_{2},...) \\
q &=&(q_{0},q_{1},q_{2},...)
\end{eqnarray*}%
are multi-indices, and%
\[
\Phi _{pq}^{(m)}=\overline{\Phi _{qp}^{(m)}}
\]%
since $U$ and $\Phi $ are real vector fields, and 
\[
\Phi _{10}^{(0)}=\Phi _{01}^{(0)}=0
\]%
since the corresponding linear terms are already taken into account in (\ref%
{decomp U}). The method consists in identifying monomials when we write (\ref%
{basic NS syst}) using the expression (\ref{decomp U}) and replace $\frac{%
\partial A}{\partial t}$ by (\ref{GL equaton}). For the derivation we need
to rewrite the operator $\mathbf{L}_{0}$ as follows%
\begin{eqnarray*}
\mathbf{L}_{0}U &=&\Delta _{\bot }U+\nabla _{\bot }p+%
\mathfrak{R}x\partial _{y}U +\left( 
\begin{array}{c}
T_{c}u_{y} \\ 
u_{x} \\ 
0%
\end{array}%
\right) , \\
0 &=&\nabla _{\bot }\cdot U_{\bot }+\mathfrak{R}\partial _{y}u_{y},
\end{eqnarray*}%
where the subscript $\bot $ means components $(x,z).$ Then, using $\mathbf{L}%
_{0}\zeta =0,$ $\mathbf{L}_{0}(A\zeta )$ becomes%
\[
\mathbf{L}_{0}(A\zeta )=\mathfrak{R}x(\partial
_{y}A)\zeta +\nabla _{\bot }p.
\]

Now, as an example, let us consider the coefficients $\Phi
_{(01)(00)}^{(0)}(x,z)$ of $\partial _{y}A$ and $\Phi
_{(002)(000)}^{(0)}(x,z)$ of $\partial _{y}^{2}A$ in $\Phi .$ Then
identifying the coefficients of $\partial _{y}A$ and defining 
\[
p=Ap_{0}+\overline{A}\overline{p_{0}}+(\partial _{y}A)p_{(01)(00)}+...
\]%
where $~(p_{0},\zeta )$ satisfy%
\begin{eqnarray*}
\Delta _{\bot }\zeta +\nabla _{\bot }p_{0}+e^{i\alpha _{c}z}\left( 
\begin{array}{c}
T _{c}u_{y}^{0} \\ 
u_{x}^{0} \\ 
0
\end{array}%
\right)  &=&0 \\
\nabla _{\bot }\cdot \zeta _{\bot } &=&0,
\end{eqnarray*}%
we obtain%
\begin{eqnarray}
0 &=&\Delta _{\bot }\Phi _{(01)(00)}+\left( 
\begin{array}{c}
T_{c}u_{y(01)(00)} \\ 
u_{x(01)(00)} \\ 
0%
\end{array}%
\right) +\mathfrak{R}x\zeta +\nabla _{\bot }p_{(01)(00)} ,  \label{Phi01} \\
0 &=&\nabla _{\bot }\cdot \Phi _{\bot (01)(00)}+ \mathfrak{R}u_{y}^{0},\text{     }\Phi
_{(01)(00)}|_{x=\pm 1/2}=0.  \nonumber
\end{eqnarray}%
For obtaining the compatibility condition we take the scalar product of (\ref{Phi01}) with $e^{i\alpha _{c}z}(u_{x}^{0},T_{c}u_{y}^{0},u_{z}^{0})^{t}.$ We observe that $\Phi _{(01)(00)}$
is uniquely determined in assuming its orthogonality with $\zeta ,$ where
the scalar product $\langle U,V\rangle $ is understood as%
\[
\langle U,V\rangle =\langle u_{x},v_{x}\rangle +\langle T _{c}u_{y},v_{y}\rangle +\langle u_{z},v_{z}\rangle .
\]%
We may observe that the coefficient of $e^{i\alpha_{c} z}$ for $u_{x(01)(00)}$ and $u_{y(01)(00)}$ are real and odd in $%
x,$ while for $u_{z(01)(00)}$ it is pure imaginary and even in $x$ and for $p_{(01)(00)}
$ it is real and even in $x.$ 

Now the identification of the coefficients of $\partial _{y}^{2}A$ leads to%
\begin{equation}
    \begin{aligned}
        -b\zeta  =&\Delta _{\bot }\Phi _{(001)(000)}+\left( 
\begin{array}{c}
T_{c}u_{y(001)(000)} \\ 
u_{x(001)(000)} \\ 
0%
\end{array}%
\right) +\mathfrak{R}x\Phi _{(01)(00)}
\\
 &+\nabla _{\bot }p_{(001)(000)}
 ,  \label{calculof b}
    \end{aligned}
\end{equation}
$$
0 =\nabla _{\bot }\cdot \Phi _{\bot (001)(000)}+\mathfrak{R}u_{y(01)(00)},\text{   }\Phi
_{(001)(000)}|_{x=\pm 1/2}. 
$$
As above we take the scalar product of (\ref{calculof b}) with $e^{i\alpha
_{z}z}(u_{x}^{0},T_{c}u_{y}^{0},u_{z}^{0})^{t},$ then we obtain $b$ in terms of $\Phi
_{(01)(00)}$ given by (\ref{Phi01}). An exercise left to the reader, consists in
verifying that we indeed obtain 
\[
b=\mathfrak{R}^{2}b_{4}.
\]

%%%%%%%%%%%%%%%%%%%%%%%%%%%%%%%%%%%%%%%%%%%%%%%%%%%%%%%%%%%%%%%%

\section{Appendix B: Numerical study \label{Appendix2} }

%%%%%%%%%%%%%%%%%%%%%%%%%%%%%%%%%%%%%%%%%%%%%%%%%%%%%%%%%%%%%%%%

%%%%%%%%%%%%%%%%%%%%%%%%

\subsection{The axisymmetric case \label{Appendix21}}

%%%%%%%%%%%%%%%%%%%%%%%%

This paragraph is devoted to the numerical study of the functions $\mathfrak{F}_1$ and $\mathfrak{F}_2$.
The figure \ref{axi} shows the value of the smallest positive zero $T_c(\alpha)$ of $\mathfrak{F}_1(\alpha,T) = 0$ as
a function of $\alpha$. For positive $\alpha$, we see that the function $T(\alpha)$ is convex and has a unique minimum
at $\alpha_c \approx 3.117$.

The zeros of $\mathfrak{F}_2(\alpha,T) = 0$ are always larger than the zeros of $\mathfrak{F}_1(\alpha,T) = 0$,
$\alpha$ being fixed.

\begin{figure}[th]
\begin{center}
\includegraphics[width=10cm]{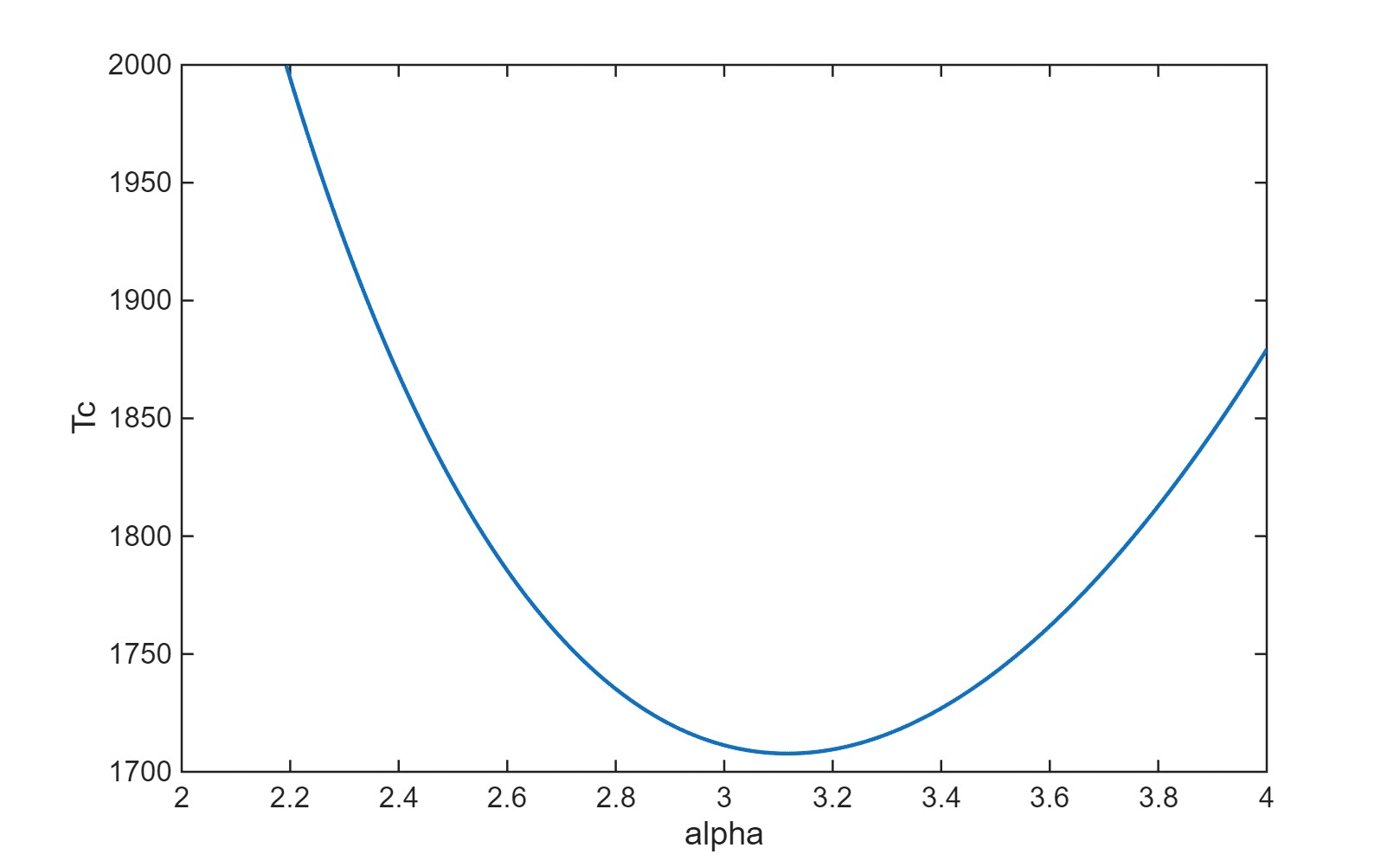}
\end{center}
\caption{Value of the smallest positive zero of $T \mapsto \mathfrak{F}_1(\alpha,T)$ as a function of $\alpha$}
\label{axi}
\end{figure}
%%%%%%%%%%%%%%%%%%%%%%%%

\subsection{The non axisymmetric case \label{Appendix22}}

%%%%%%%%%%%%%%%%%%%%%%%%%

In this section, we discuss the numerical study of the following system
\begin{eqnarray*}
(\lambda +\alpha ^{2}-D^{2})(\alpha ^{2}-D^{2})u_{x}+i\mathfrak{B}%
x(D^{2}-\alpha ^{2})u_{x} &=&T\alpha ^{2}{u}_{y} , \\
(\lambda +\alpha ^{2}-D^{2}) u_{y}-i\mathfrak{B}x u_{y} &=&u_{x},
\end{eqnarray*}%
where the parameters are $\alpha ^{2},T,\mathfrak{B}$, and where the boundary
conditions are%
\[
u_{x}= u_{y}=Du_{x}=0,x=\pm 1/2.
\]
First, we rewrite this system as a first order system on 
$$
U = ( u_x, D u_x, D^2 u_x, D^3 u_x, u_y, D u_y)
$$
which leads to 
\begin{align*}
D^4 u_x &= 
- (\lambda +\alpha ^{2}-i \mathfrak{B} x)(\alpha ^{2}-D^{2})u_{x}
+ D^2 \alpha^2 u_x + T\alpha ^{2} u_{y}, 
\\
D^2 u_y &= (\lambda +\alpha ^{2}- i \mathfrak{B} x)u_{y} - u_x,
\end{align*}
a system that we shorten in
$$
D U = A U
$$
where $A$ is the corresponding $6 \times 6$ matrix.

We consider this equation as a Cauchy problem starting at $x = -1/2$. We construct three independent
solutions $U^1$, $U^2$ and $U^3$ defined by
$$
U^1(-1/2) = (0,0,1,0,0,0), \quad
U^2(-1/2) = (0,0,0,1,0,0)
$$
and
$$
U^3(-1/2) = (0,0,0,0,0,1).
$$
Any solution to our problem is a linear combination of $u^1$, $u^2$ and $u^3$ which satisfied
$$
\delta(\lambda,\alpha,B,T) := \left| \begin{array}{ccc}
u_x^1(1/2) & u_x^2(1/2) & u_x^3(1/2) \cr
Du_x^1(1/2) & Du_x^2(1/2) & Du_x^3(1/2) \cr
u_y^1(1/2) & u_y^2(1/2) & u_y^3(1/2) \cr
\end{array} \right| = 0.
$$
It is very easy to approximate numerically $u^1$, $u^2$ and $u^3$ and to compute $\delta$.

\begin{figure}[th]
\begin{center}
\includegraphics[width=10cm]{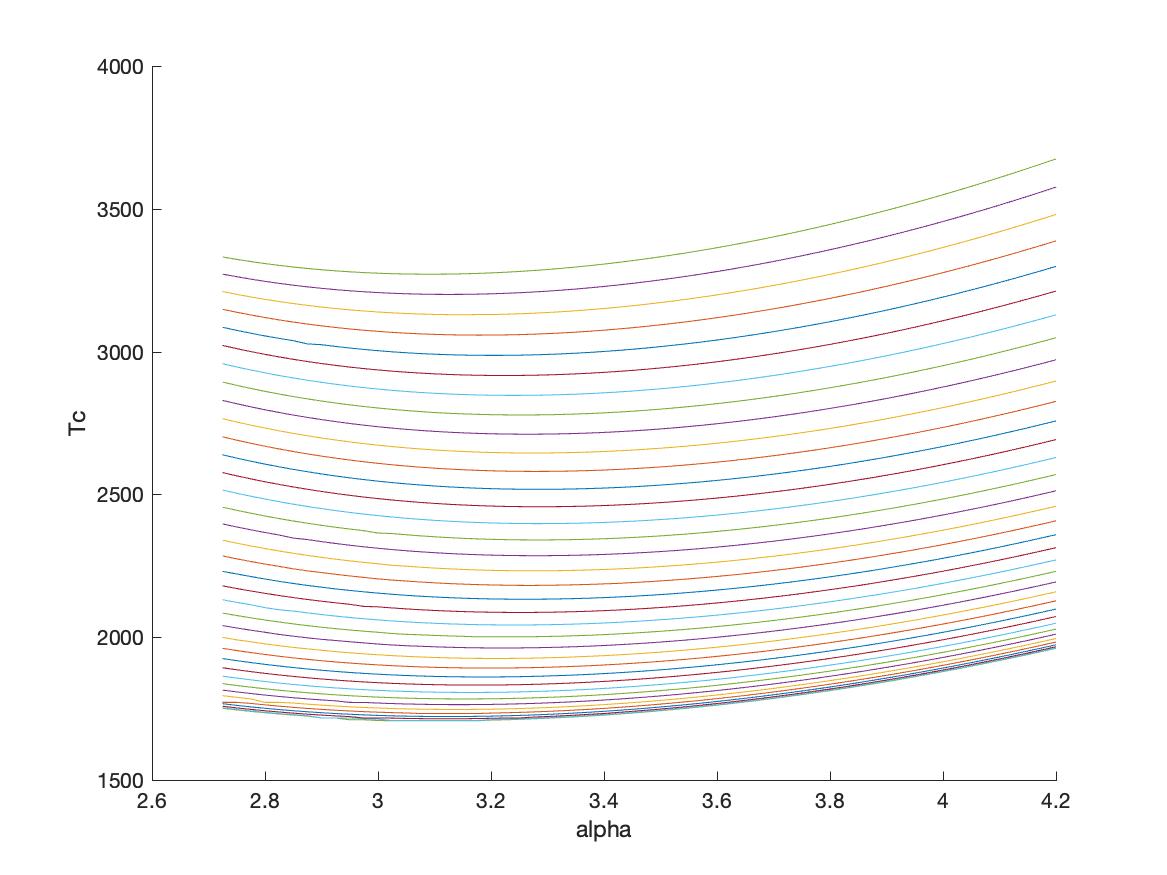}
\end{center}
\caption{The critical Taylor number $T_c(\alpha,\beta)$ as a function of $\alpha$ for various values of 
$\mathfrak{B}$.
}
%\label{alphaTc}
\end{figure}

\begin{figure}[th]
\begin{center}
\includegraphics[width=10cm]{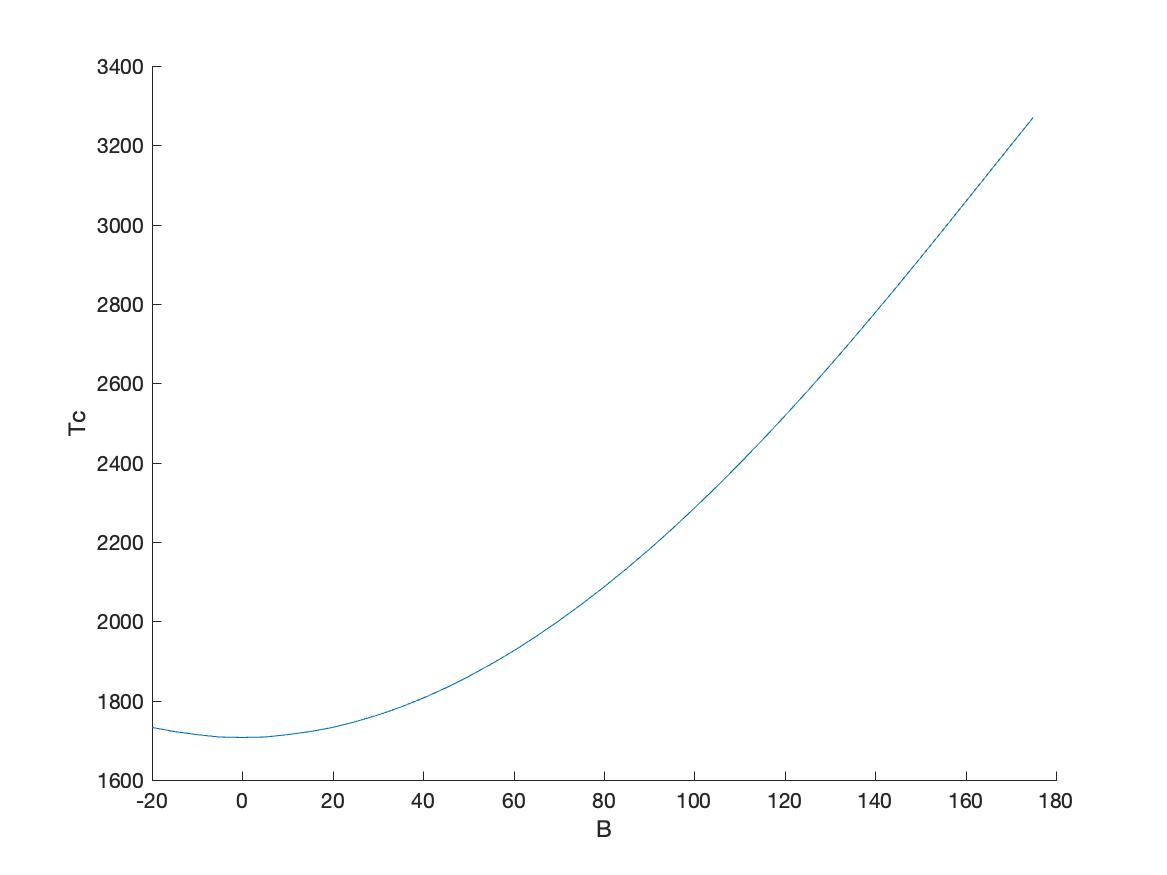}
\end{center}
\caption{The critical Taylor number $T_c$ as a function of $\mathfrak{B}$ }
%\label{alphaTc}
\end{figure}

\begin{figure}[th]
\begin{center}
\includegraphics[width=10cm]{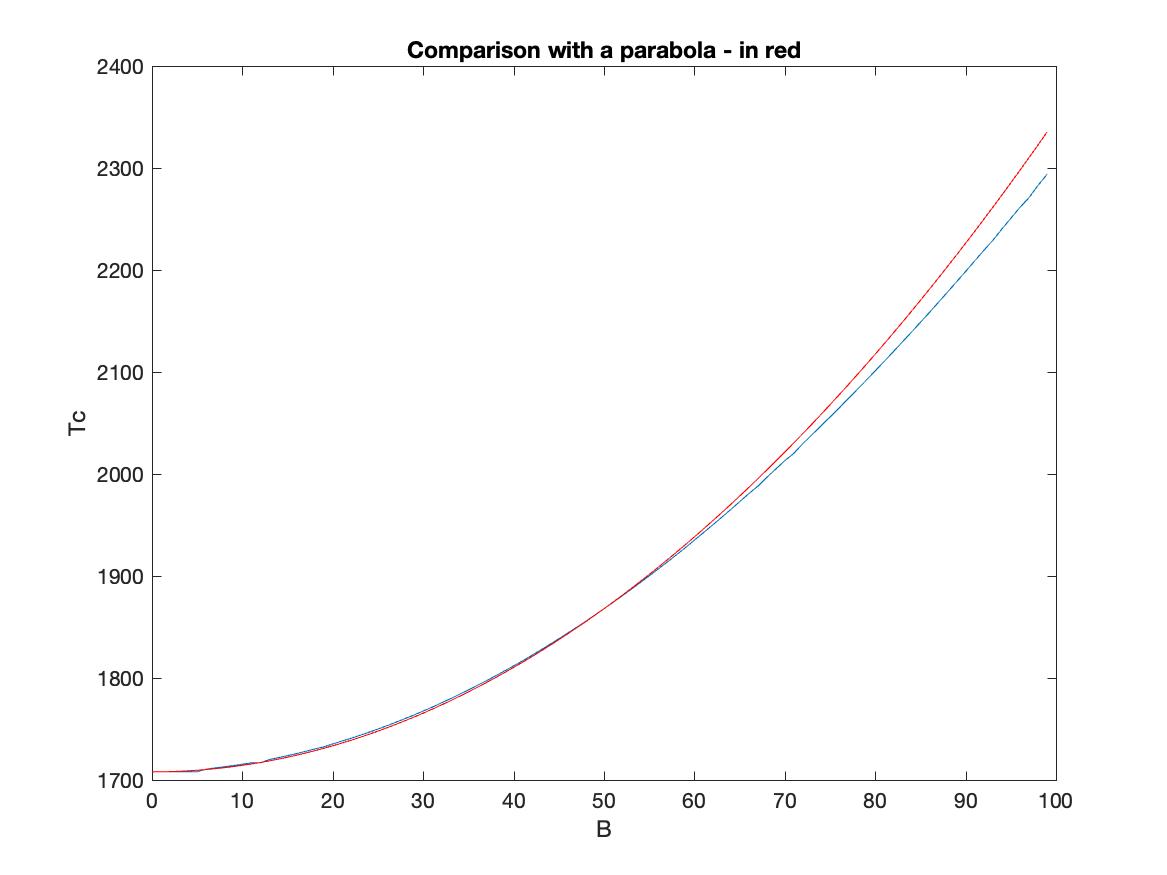}
\end{center}
\caption{The critical Taylor number $T_c(\alpha_c,\mathfrak{B})$ as a function of $\mathfrak{B}$ (blue curve) and its quadratic approximation by 
$T_c + 0.064 \mathfrak{B}^2$.}
\label{comparison}
\end{figure}

The first step is to solve $\delta(T) = 0$, $\alpha$ and $\mathfrak{B}$ being given.
This gives the critical Taylor number $T_c(\alpha,\mathfrak{B})$.
The second step is to compute the minimum $T_c(\mathfrak{B})$ of $T_c(\alpha,\mathfrak{B})$ over $\alpha$.

We note that $T_c(\mathfrak{B})$ is minimum when $\mathfrak{B} = 0$, where it equals the classical value of $1707$
for $\alpha \approx 3.117$.
Moreover, on the range of $\alpha$ which is studied on this figure, $T_c$ is a convex function of $\alpha$ when
$\mathfrak{B}$ is fixed.

Figure \ref{comparison} displays $T_c(\alpha_c,\mathfrak{B})$ as a function of $\mathfrak{B}$, together with its quadratic approximation at $\mathfrak{B} = 0$. This in particular shows, numerically, that $a_4 < 0$.

%%%%%%%%%%%%%%%%%%%%%%%%%%%%%%%%%%

\subsubsection*{Acknowledgments}  

The authors would like to warmly thank Y. Guo for many fruitful discussions.
D. Bian is supported by NSFC under the contract 12271032.

\subsubsection*{Conflict of interest}  The authors state that there is no conflict of interest.

\subsubsection*{Data availability}  Data are not involved in this research paper.

%%%%%%%%%%%%%%%%%%%%%%%%%%%%

\end{document}